\documentclass[11pt,reqno,tbtags]{amsart}
\usepackage{amssymb}
\usepackage{amsmath}
\usepackage{amscd}
\usepackage{amsthm}
\usepackage[swedish, english]{babel}
\usepackage[utf8]{inputenc}
\usepackage{psfrag}   
\usepackage{graphicx}
\usepackage[active]{srcltx}

\numberwithin{equation}{section}

\theoremstyle{plain}
\newcommand{\ep}{\epsilon}
\newcommand{\R}{\mathbb R}
\newcommand{\Z}{\mathbb Z}

\newtheorem{theorem}{Theorem}[section]
\newtheorem{lemma}[theorem]{Lemma}

\theoremstyle{definition}
\newtheorem{example}[theorem]{Example}
\newtheorem{remark}[theorem]{Remark}
\newtheorem{remarks}[theorem]{Remarks}
\newtheorem*{ack}{Acknowledgement}

\newenvironment{romenumerate}[1][0pt]{
\addtolength{\leftmargini}{#1}\begin{enumerate}
 }{\end{enumerate}}

\newcounter{oldenumi}
\newenvironment{romenumerateq}
{\setcounter{oldenumi}{\value{enumi}}
\begin{romenumerate} \setcounter{enumi}{\value{oldenumi}}}
{\end{romenumerate}}

\newcommand{\refT}[1]{Theorem~\ref{#1}}

\newcommand{\refL}[1]{Lemma~\ref{#1}}
\newcommand{\refE}[1]{Example~\ref{#1}}
\newcommand{\refR}[1]{Remark~\ref{#1}}
\newcommand{\refRR}[1]{Remarks~\ref{#1}}
\newcommand{\refS}[1]{Section~\ref{#1}}

\newcommand\marginal[1]{\marginpar{\raggedright\tiny #1}}

\newcommand\REM[1]{{\raggedright\texttt{[#1]}\par\marginal{XXX}}}

\newcommand\ooo{[0,\infty)}
\newcommand\ooox{[0,\infty]}

\newcommand\intoooo{\int_{-\infty}^\infty}
\newcommand\downto{\searrow}
\newcommand\upto{\nearrow}

\newcommand\gd{\delta}
\newcommand\gD{\Delta}
\newcommand\gf{\varphi}

\newcommand\gG{\Gamma}

\newcommand\gL{\Lambda}
\newcommand\go{\omega}

\newcommand\gs{\sigma}
\newcommand\gss{\sigma^2}
\newcommand\eps{\varepsilon}

\newcommand\cE{\mathcal E}
\newcommand\cF{\mathcal F}
\newcommand\cG{\mathcal G}

\newcommand\cL{{\mathcal L}}

\newcommand\bigpar[1]{\bigl(#1\bigr)}

\newcommand\lrpar[1]{\left(#1\right)}

\newcommand\xcpar[1]{\{#1\}}

\newcommand\set[1]{\ensuremath{\{#1\}}}
\newcommand\bigset[1]{\ensuremath{\bigl\{#1\bigr\}}}
\newcommand\Bigset[1]{\ensuremath{\Bigl\{#1\Bigr\}}}

\newcommand\supp{\operatorname{supp}}
\newcommand\suppoo{\operatorname{supp_\infty}}
\newcommand\E{\operatorname{\mathbb E{}}}
\renewcommand\P{\operatorname{\mathbb P{}}}
\newcommand\Var{\operatorname{Var}}
\newcommand\ntoo{\ensuremath{{n\to\infty}}}
\newcommand\ttoo{\ensuremath{{t\to\infty}}}
\newcommand\utoo{\ensuremath{{u\to\infty}}}
\newcommand{\tend}{\longrightarrow}
\newcommand\dto{\overset{\mathrm{d}}{\tend}}

\newcommand\bigabs[1]{\bigl|#1\bigr|}
\newcommand\ett[1]{\boldsymbol1\xcpar{#1}}

\newcommand\xo{x_0}
\newcommand\yo{y_0}
\newcommand\bmu{\overline\mu}
\newcommand\intR{\int_{\R}}
\newcommand\nbh{neighbourhood}
\newcommand\cc{C_{\mathrm c}}
\newcommand\ccp{\cc^+}
\newcommand\dd{\,\mathrm{d}}
\newcommand\bbm{[b_-^m,b_+^m]}
\newcommand\aax{(a_-,a_+)}
\newcommand\tnu{\widetilde\nu}
\newcommand\floor[1]{\lfloor#1\rfloor}

\newcommand\norm[1]{\|#1\|}
\newcommand\normq[1]{\norm{#1}_{2}}
\newcommand\atm{{A_{t-}}}
\newcommand\aoo{{A_{\infty}}}
\newcommand\bn{B^n}
\newcommand\bni{B^{n-1}}
\newcommand\pin{\Pi^n}
\newcommand\pini{\Pi^{n-1}}
\newcommand\gdni{\gD_{n+1}}
\newcommand\mxo{M_{\xo}}
\newcommand\sni{S^{n-1}}
\newcommand\bd{\partial}
\newcommand\bb{(b_i)_1^n}
\newcommand\ppi{(\pi_i)_0^{n+1}}
\newcommand\tH{\tilde H}
\newcommand\kk{\kappa}

\newcommand\xs{x}
\newcommand\xsoo{x^\infty}
\newcommand\xx{\xi}
\newcommand\HH{{\tilde H}}
\newcommand\hoo{H_{\suppoo\nu}}

\begingroup
  \divide\count255 by 60
  \multiply\count255 by -60
  \advance\count255 by \time
  \ifnum \count255 < 10 \xdef\klockan{\the\count1.0\the\count255}
  \else\xdef\klockan{\the\count1.\the\count255}\fi
\endgroup

\title{Can time-homogeneous diffusions
produce any distribution?}
\author{Erik Ekstr\"om}
\address{Department of Mathematics, Uppsala University, Box 480, SE-751 06
  Uppsala, Sweden}
\email{erik.ekstrom@math.uu.se}

\author{David Hobson}
\address{Department of Statistics,
Zeeman Building,
University of Warwick,
Coventry, CV4 7AL, UK}
\email{D.Hobson@warwick.ac.uk}

\author{Svante Janson}
\address{Department of Mathematics, Uppsala University, Box 480, SE-751 06
  Uppsala, Sweden}
\email{svante.janson@math.uu.se}

\author{Johan Tysk}
\address{Department of Mathematics, Uppsala University, Box 480, SE-751 06 Uppsala, Sweden}
\email{johan.tysk@math.uu.se}

\date{March 21, 2011} 

\begin{document}

\begin{abstract} 
Given a centred distribution, can one find a time-homogeneous martingale 
diffusion starting at zero which has the given law at time 1?
We answer the question affirmatively if generalized diffusions are allowed.
\end{abstract}

\maketitle

\section{Introduction}

Consider a distribution with finite mean on the real line. Can this
distribution be recovered as the distribution of a diffusion
at time 1? Clearly not if the support of the distribution
is disconnected since diffusions are continuous, but the answer 
is yes, 
and in many ways, if the target distribution is sufficiently regular.
What if we also require our process to be a martingale or to be 
time-homogeneous? Again, easy constructions exist (see 
Remark~\ref{rem:nonexample} 
below) to show that a suitable process exists. But, what if we require 
our process to be both time-homogeneous and a martingale?
The original motivation for us to study this problem comes from a 
calibration problem in mathematical finance.

Our approach involves the 
speed measure of a 
diffusion and time-changes of Brownian motion, and makes no assumptions 
on the regularity of the target law. Indeed, to allow for target 
distributions with arbitrary support the natural class of processes to 
consider is the class of generalized diffusions (sometimes referred to 
as gap diffusions) discussed below. Our main result is to show that 
within this class of processes there is a time-homogeneous martingale 
with the given distribution at time 1.

In Section~\ref{construction} we introduce the class of generalised diffusions 
as time-changes of a Brownian motion. We also formulate our 
main result, Theorem~\ref{cont}, which states that given a distribution on 
$\R$
with finite mean, there exists a generalised diffusion that is a martingale and that has this
distribution at time 1. In Section~\ref{Sprel} we collect some general results
about time-changes and generalised diffusions.
In Section~\ref{discrete} we study a discrete version of the inverse problem.
For a given distribution with mass only in a finite number of points, we show that 
there exists a time homogeneous martingale Markov chain with the 
given distribution at time 1.
In Section~\ref{general} we consider the case of a general distribution
on the real axis. By approximating the distribution with finitely supported 
distributions, the existence of a solution to the inverse problem is
obtained, thereby  
proving Theorem~\ref{cont}. 
In Section~\ref{finance}, we apply our results to solve 
a calibration problem in mathematical finance. In Section~\ref{martingality}
we characterize the speed measures for which the corresponding generalised 
diffusion is a (local) martingale.
Finally, Section~\ref{openproblems} 
concludes with some open problems.

\section{Construction of generalised diffusions}
\label{construction}
In this section we construct, following \cite{EH} (see also \cite{K} and \cite{KW}), 
time-homogeneous generalised diffusion processes as time-changes 
of Brownian motion. The time-change is specified in terms of the so-called
\emph{speed measure}. 

Let $\nu$ be a nonnegative Borel measure on the real line;
$\nu$ may be finite or infinite.
Let $B_t$ be a Brownian
motion  
starting at $\xo$, and let $L_u^x$ be its local time at the point $x$ up to
time $u$. 
Recall that (a.s.) $L_u^x$ is continuous in
$(u,x)\in\ooo\times\R$ and increasing in $u$. 
Define the increasing process
\begin{equation}
\label{gamma}
\Gamma_u:=\int_\R L_u^x \nu(dx),
\end{equation}
noting that  $\gG_u\in\ooox$,
and let 
\begin{equation}\label{at}
A_t:=\inf\{u: \Gamma_u > t\}  
\end{equation}
be its right-continuous inverse.
The process
\[X_t:=B_{A_t}\]
will be called a generalised diffusion with speed measure $\nu$.

\begin{example}\label{Ediffusion}
If $\nu(dx)=\frac{dx}{\sigma^2(x)}$ for some continuous
non-vanishing 
function $\sigma$, then $X_t$ is a weak solution of
\[dX_t=\sigma(X_t)\,dW_t.\]
In this case, $X_t$ is a diffusion.
This example is the motivation for calling the measure $\nu$ `speed measure',
but note that $\nu$ rather measures the inverse of the speed.
\end{example}

\begin{remarks}\label{RRdef} 
  \begin{romenumerate}
  \item 
Almost surely, $L_u^x\upto \infty$ for every $x$ as \utoo; 
hence
if $\nu$ is non-zero, then $\gG_u\to\infty$ a.s.\ as \utoo, 
and thus $A_t<\infty$ and $X_t$ is well-defined for every $t\in\ooo$ a.s.
However, we have to exclude the exceptional case $\nu=0$, when $\gG_u=0$ for
every 
$u$ and $A_t=\infty$ for every $t\ge0$, so $X_t$ is not defined.
(For technical reasons, we allow $\nu=0$ when discussing $A_t$, but we
always assume $\nu\neq0$ when considering $X_t$, sometimes without explicitly
saying so.)

\item \label{RRgGcont}
$\gG_u$ is left-continuous (by monotone convergence) and
continuous at every $u$ such that $\gG_{u+}<\infty$ (by dominated
convergence); hence $u\mapsto\gG_u$ is continuous everywhere except that
there might exist a single (random) finite $u_0$ where $\gG_u$ jumps to
$+\infty$: $\gG_{u_0}<\infty$ but $\gG_u=\infty$ for all $u>u_0$.
(For example, this happens if $\nu$ is an infinite point mass, see
\refE{Epoint} below.)

\item \label{RR<} 
By \eqref{at} and the left-continuity of $\gG_u$, for all $t,u\ge0$,
\begin{equation}\label{rr<}
  A_t<u \iff \gG_u>t.
\end{equation}
Equivalently, 
$  u\le A_t \iff t\ge\gG_u$.
It follows that $\gG_{A_t}\le t \le \gG_{A_t+}$ for all $t$, and thus, 
by \ref{RRgGcont},  
$\gG_{A_t}=t$ for all $t$ such that $\gG_{A_t+}<\infty$; see further
\refL{Lgat}.

\item 
$\gG_u$ is finite for all $u$ a.s.\
if and only if $\nu$ is a locally finite measure
(also called a \emph{Radon measure}), 
i.e., $\nu(K)<\infty$ for every compact $K$. 
In this case, $\gG_u$ is continuous by \ref{RRgGcont}, and  $\gG_u$ is a
\emph{continuous additive functional (CAF)} of $B$;
conversely, every CAF of Brownian motion is of this type, see
\cite[Theorem 22.25]{Kallenberg} and \cite[Chapter X]{RY}.

\item \label{RRdisc}
Although $B_0=\xo$, $A_0$ may be strictly positive, and in general
$X_0\neq\xo$, see \refL{L0} below for a precise result. 
In \cite{EH}, $X_0$ is defined to be $\xo$, but this has the
disadvantage of making $X_t$ possibly not right-continuous at $t=0$.
We follow here instead
the standard practise of considering right-continuous processes.
We define 
$A_{0-}=0$ and
$X_{0-}=B_{A_{0-}}=\xo$, and thus allow the possibility that $X_{0-}\neq X_0$.

\item 
We let $(\cF_t)_{t\ge0}$ denote the standard completed Brownian filtration.
Then each $A_t$ is a $(\cF_t)$-stopping time, and $X_t$ is adapted
to the filtration $(\cG_t)=(\cF_{A_t})$, ${t\ge0-}$.
(In particular, $\cG_{0-}=\cF_0$ is trivial.)

In the sequel, we let ``stopping time'' mean $(\cF_t)$-stopping time unless
we say otherwise.

\item 
Even though the process $X_t$ is constructed as a time change of Brownian
motion, it is in general not necessarily a martingale or even a local
martingale; see \refS{martingality} for a detailed discussion. 
However, we are mainly interested in cases in which
$X_t$ is a local martingale, and preferably an integrable martingale.  
Recall the convention just made that $X_{0-}=\xo$ while $X_0$ may be
different. We say that $(X_t)_{t\ge0-}$ is a (local) martingale if 
$(X_t)_{t\ge0}$ is a (local) martingale 
(for the filtration $(\cG_t)$) 
and, further, $\E X_0=X_{0-}$.
(This is equivalent to the standard definition interpreted for the index set
$\set{0-}\cup[0,\infty)$.)  
\end{romenumerate}
\end{remarks}

Our main result is that any given distribution with finite mean 
can be obtained as the distribution of $X_1$ for some such
generalised diffusion with a suitable choice of speed measure $\nu$.

\begin{theorem}
\label{cont}
Let $\mu$ be a distribution on the real axis with finite mean $\bmu=\int_\R
x\mu(dx)$. 
Then there exists a 
generalised diffusion $X$ such that $X_{0-}=\overline \mu$,
$(X_t)_{0-\le t\le 1}$ is a martingale,
and the distribution of $X_1$ is $\mu$.
Furthermore: 
\begin{romenumerate}
\item \label{cont-X_0} 
$X_0=\bmu$ if and only if $\bmu\in\supp\mu$.
\item \label{cont-var}
$\E A_1=\Var(\mu):=\intR (x-\bmu)^2\mu(dx)\le\infty$. 
In particular, $\E A_1<\infty$ if and only if $\Var(\mu)<\infty$.
\end{romenumerate}
\end{theorem}

It follows from \refT{Tmartingale}
that, actually,
$(X_t)_{0-\le t<\infty}$ 
is a martingale.

\begin{remark} 
\label{rem:nonexample} 
Theorem~\ref{cont} guarantees the existence of a 
time-homogeneous (generalised) diffusion which is a martingale and has a 
certain distribution at time 1. Note that the problem is much easier if 
some of these requirements are dropped. For example, for a given 
distribution $\mu$, one can find a non-decreasing function $f:\R\to\R$ 
such that $f(B_1)$ has law $\mu$ (in particular, $f = 
F_{\mu}^{-1} \circ \Phi$, where $F_\mu$ is the cumulative distribution 
function associated with $\mu$ and $\Phi$ is the normal CDF) and then 
the process $X_t=f(B_t)$ is time-homogeneous with distribution $\mu$ at 
time 1. Further, if $\mu$ has a density which is strictly 
positive and differentiable, then $X$ is a time-homogeneous diffusion 
with $dX_t = a(X_t)\,dB_t + b(X_t) dt$ where $a := f'\circ f^{-1}$ 
and $2b:= f'' \circ f^{-1}$. Typically, however, it is not a martingale.

Similarly, with $f$ as above, the process $M_t=\E (f(B_1)\mid \cF_t)$ is a 
martingale with distribution $\mu$ at time 1, but typically it is not 
time-homogeneous. Again, in the regular case, we have 
$M_t=g(B_t,t)$ for some function $g$ solving the heat equation, 
and $dM_t = g'(g^{-1}(M_t,t),t)\,dB_t$ is a time-inhomogeneous martingale
diffusion, where  
$g^{-1}(\cdot,t)$ is the space-inverse of 
$g(\cdot, t)$, and $g'(\cdot,t)$ denotes the spatial derivative.
\end{remark}

\begin{remark}
In \cite{JT}, a similar inverse problem is studied. Instead of allowing for 
generalised diffusions, the authors consider the case of diffusions 
with regular diffusion coefficients obtaining an approximate solution 
to the inverse problem using variational techniques. In \cite{CHO} the authors 
solve a related inverse problem in which the goal is to construct a 
generalised diffusion $Y$ such that $Y$ is a martingale and such that 
the distribution of $Y_{\tau_E}$ is $\mu$, where $\tau_E$ is an independent 
random exponential time. Stopping at an independent exponential time is more 
tractable than stopping at a fixed time.
\end{remark}

\begin{remark}
One interpretation of our results is that we construct a stopping time 
$\tau \equiv A_1$ such that $B_{\tau} \equiv B_{A_1} \equiv X_1$ has law 
$\mu$.
The general problem of finding a stopping time $T$ 
such that $B_T$ has a
given distribution is known as the \emph{Skorokhod stopping problem}, see
e.g.\ \cite[Section 5.3]{MPeres}, where other constructions of such
stopping times are given. 
\end{remark}

\begin{example}
As a very simple case, if $\mu$ is the normal distribution $N(\xo,\gss)$,
we may take $d\nu=\gs^{-2} dx$, a constant multiple of the Lebesgue measure.
Then $\gG_u=\gs^{-2} u$ and $A_t=\gss t$ (both non-random), so
$X_t=B_{\gss t}$, cf.\ \refE{Ediffusion}.
In general, however, it seems difficult to find $\nu$ explicitly.
\end{example}

\section{Preliminaries}\label{Sprel}

Recall that the \emph{support} of a measure $\nu$ on $\R$ is 
\begin{equation*}
\supp\nu:=\R\setminus\bigcup\bigset{U\subseteq\R:U\text{ is open and }\nu(U)=0}.
\end{equation*}
In other words, $x\in\supp\nu$
if and only if every neighbourhood of $x$
has positive measure.

Similarly, we define the \emph{infinity set} 
$\suppoo\nu$ by
\begin{equation*}
\suppoo\nu
:=\R\setminus\bigcup\bigset{U\subseteq\R:U\text{ is open and }\nu(U)<\infty}.
\end{equation*}
In other words, $x\in\suppoo\nu$ if and only if every \nbh{} of $x$
has infinite measure.
Thus, $\suppoo\nu=\emptyset$ if and only if $\nu$ is locally finite. 
By definition, $\supp\nu$ and $\suppoo\nu$ are closed subsets of $\R$.

If $S$ is any Borel set in $\R$, we let $H_S:=\inf\set{t\ge0:B_t\in S}$
denote the \emph{hitting time} of $S$ (for the Brownian motion $B_t$).
Note that if $S\neq\emptyset$, then $H_S<\infty$ a.s.
We will only consider cases when $S$ is closed, and then $B_{H_S}\in S$.
For $x\in\R$, we write $H_x$ for $H_{\set x}$. 

\begin{lemma}
  \label{Lsupp}
If $\nu\neq0$, then a.s.\ $X_t=B_{A_t}\in\supp \nu$ for all $t\ge0$.
\end{lemma}
\begin{proof}
  If $s>0$ and $B_s\notin\supp\nu$, then there exists $\eps>0$ such that
  $B_u\notin\supp\nu$ 
for $u\in[s-\eps,s+\eps]$, and then $L_u^x$ is constant
  for $u\in[s-\eps,s+\eps]$ for each $x\in\supp\nu$; hence $\gG_u$ is
  constant for $u\in[s-\eps,s+\eps]$, and thus $A_t\neq s$ for all $t\ge 0$.

Similarly, if $s=0$ and $B_s\notin\supp\nu$, then there exists $\eps>0$
such that $\gG_u=0$ for $u\in[0,\eps]$. Therefore $A_t>0=s$ for all $t\geq 0$.
\end{proof}

\begin{lemma}
  \label{LT}
If $T$ is any finite stopping time for $B_t$, then a.s., for all $\eps>0$,
\begin{equation*}
L_{T+\eps}^{B_T} > L_{T}^{B_T}  \ge0.
\end{equation*}
Consequently, there is for every $\eps>0$ 
a.s.\ a (random) $\gd>0$ and an open set $U$
containing $B_T$ such that
$L_{T+\eps}^{x} - L_{T}^{x}  \ge\gd$ for all $x\in U$.
\end{lemma}
\begin{proof}
The first claim is
  an immediate consequence of the strong Markov property and the fact that
  $L^x_\eps>0$ a.s.{} for a Brownian motion started at $x$ and any $\eps>0$.

The second claim follows since $L$ is continuous.
\end{proof}

\begin{lemma}
  \label{L0}
$A_0=H_{\supp\nu}$, the hitting time  of the support of $\nu$, a.s. 
In particular, if $\xo\in\supp\nu$, then $A_0=0$ and $X_0=\xo$ a.s.,
but if $\xo\notin\supp\nu$, then $A_0>0$ and 
$X_0\neq\xo$ a.s.
\end{lemma}

\begin{proof}
  If $u<H=H_{\supp\nu}$, then $\gG_u=0$.
On the other hand, by \refL{LT},
for any $\eps>0$, $L_{H+\eps}^{B_H}>0$ a.s.\ 
and, moreover,
$L_{H+\eps}^{x}>0$ for $x$ in a \nbh{} $U$ of $B_H$.
Since $B_H\in\supp\nu$, we have $\nu(U)>0$ and thus
$\gG_{H+\eps}>0$. In other words, a.s.\ $\gG_u>0$ for all $u>H$.

The definition \eqref{at} of $A_t$ now shows that $A_0= H$ a.s., and the
result follows, recalling \refL{Lsupp}.
\end{proof}

\begin{lemma}
  \label{L00}
We have $\E X_0=x_0=X_{0-}$ if and only if neither
$\supp\nu\subset(-\infty,\xo)$ nor
$\supp\nu\subset(\xo,\infty)$.
\end{lemma}
\begin{proof}
  If $\supp\nu$ intersects both  $[\xo,\infty)$ 
and 
$(-\infty,\xo]$,
then 
$B_{u\wedge H_{\supp\nu}}$
is a bounded martingale and thus $\E X_0=\E B_{H_{\supp\nu}}=B_0=\xo$.  

On the other hand, if $\supp\nu\subset(-\infty,\xo)$, then \refL{Lsupp}
implies $X_0<x_0$ a.s., so $\E X_0<\xo$. Similarly, if
$\supp\nu\subset(\xo,\infty)$, then $\E X_0>\xo$.
\end{proof}

We define, for given $\nu$ and $\xo$, 
\begin{align}
\xs_+&:=\inf\set{x\ge x_0:x\in\supp\nu}, \label{xs+}\\
\xs_-&:=\sup\set{x\le x_0:x\in\supp\nu},\label{xs-}\\
\xsoo_+&:=\inf\set{x\ge x_0:x\in\suppoo\nu}, \label{xsoo+}\\
\xsoo_-&:=\sup\set{x\le x_0:x\in\suppoo\nu}. \label{xsoo-}
\end{align}
Note that these may be $\pm\infty$ (when the corresponding sets are empty).
In general,
\begin{equation*}
  -\infty \le\xsoo_-\le \xs_-\le\xo\le\xs_+\le\xsoo_+\le\infty.
\end{equation*}

It follows from Lemmas \ref{L0} and \ref{L00}
that we have the following cases:
\begin{romenumerate}
\item 
If $\xo\in\supp\nu$ (i.e., $\xs_-=\xs_+=\xo$), 
then $X_0=\xo$.
\item 
If $\supp\nu\subset (\xo,\infty)$ 
(i.e., $\xs_-=-\infty$ and $\xo<\xs_+<\infty$), 
then $X_0=\xs_+$.
\item 
If $\supp\nu\subset (-\infty,\xo)$  
(i.e., $\xs_+=+\infty$ and $-\infty<\xs_-<\xo$),
then $X_0=\xs_-$.
\item 
Otherwise (i.e., if $-\infty<\xs_-<\xo$ and $\xo<\xs_+<\infty$), 
then $X_0\in\set{\xs_-,\xs_+}$, with the unique
distribution satisfying $\E X_0=\xo$.
\end{romenumerate}

\begin{lemma}
  \label{LH}
Let $H=H_{\suppoo\nu}$ be the hitting time for $B_t$ of the infinity set of
$\nu$. Then a.s.\ $\gG_u=\infty$ for all $u>H$, and thus $A_t\le H$ for all
$t\ge0$. 

On the other hand, $\gG_u<\infty$ a.s.\ for all $u<H$.
\end{lemma}

\begin{proof}
  If $H<\infty$, then $B_H\in\suppoo\nu$.
For any $\eps>0$, 
 by \refL{LT},
$L_{H+\eps}^{x}\ge\gd>0$ for $x$ in a \nbh{} $U$ of $B_H$  and
some $\gd>0$. 
Since $B_H\in\suppoo\nu$, we have $\nu(U)=\infty$ and thus
$\gG_{H+\eps}\ge \int_U L_{H+\eps}^{x}\nu(dx)=\infty$.
The definition \eqref{at} of $A_t$ now shows that $A_t\le H$ a.s.

If $u<H$, then $K:=\set{B_s:0\le s\le u}$ is a compact interval disjoint from
$\suppoo\nu$, and thus $\nu(K)<\infty$. Since $x\mapsto L_u^x$ is continuous
and vanishes outside $K$, we have $\gG_u=\int_K L_u^x\nu(dx)<\infty$.
\end{proof}

\begin{example}\label{Epoint}
Let $\nu$ be an infinite point mass at $x_1\in\R$. Then
$\supp\nu=\suppoo\nu=\set{x_1}$. If $H_{x_1}$ is the hitting time of $x_1$,
then $\gG_u=0$ for $u< H_{x_1}$,  cf.\ \refL{L0} and its proof,
but a.s.\ $\gG_u=\infty$ for $u>H_{x_1}$ by \refL{LH}.
Hence, a.s., $A_t=H_{x_1}$ and $X_t=x_1$ for every $t\ge0$.

In particular, if $x_1=x_0$, then $H_{x_1}=0$ a.s., and thus $A_t=0$
a.s.\ for all $t\ge0$.

More generally, if $\nu$ is a measure such that
$\nu(S)=0$ or   $\nu(S)=\infty$ for every Borel set $S$, then 
$\suppoo\nu=\supp\nu$, and Lemmas \ref{LH} and \ref{L0}
imply that a.s.\ $A_t=H_{\supp\nu}$ for all $t\ge0$, so $X_t=X_0$ 
is the first point of $\supp\nu$ hit by $B_u$.
\end{example}

\begin{lemma}\label{LB} 
Suppose that $\suppoo\nu\cap(-\infty,\xo]$ and $\suppoo\nu\cap[\xo,\infty)$
are both non-empty. Then  $B_{A_t\wedge u}\in[\xsoo_-,\xsoo_+]$ for all
$t,u\ge0$, and thus $(B_{A_t\wedge u})_{u\ge 0}$ is a bounded martingale for
each fixed $t$.  
\end{lemma}

\begin{proof}
  Recall $\xsoo_\pm$ from \eqref{xsoo+}--\eqref{xsoo-}.
By assumption,
  $-\infty<\xsoo_-\le\xo\le\xsoo_+<\infty$. 
Let $H=H_{\suppoo\nu}=H_{\set{\xsoo_-,\xsoo_+}}$. 
Then $\xsoo_-\le B_u\le \xsoo_+$ for
all $u\le H$; thus \refL{LH} implies that 
$\xsoo_-\le B_{u\wedge A_t}\le \xsoo_+$ for
any $u\ge 0$ and $t\ge0$.  
Finally, $(B_{A_t\wedge u})_{u\ge 0}$ is a martingale since $A_t$ is a
stopping time.
\end{proof}

\begin{lemma}
  \label{LWald} 
If $\E A_t<\infty$, then 
$\E X_t=\xo$ and $\Var X_t=\E (X_t-\xo)^2=\E A_t<\infty$.
\end{lemma}
\begin{proof}
  This is an instance of Wald's lemmas, see e.g.\ 
\cite[Theorems 2.44 and 2.48]{MPeres}. 
\end{proof}

\begin{lemma}
  \label{Lmart} 
If $\E A_{t_0}<\infty$ for some $t_0<\infty$,
then $(X_t)_{0-\le t\le t_0}$ is a square integrable martingale.
\end{lemma}
\begin{proof}
Since each $A_t\wedge n$ is a bounded stopping time, 
if $0-\le s<t$, then 
\begin{equation}
  \label{lmart}
\E (B_{A_t\wedge n}\mid \cF_{A_s}) = B_{A_t\wedge n\wedge A_s}
= B_{A_s\wedge n}
\end{equation} 
a.s.\ (see e.g.\ \cite[Theorem 7.29]{Kallenberg}), so 
$B_{A_t\wedge n}$, $t\ge0-$, is a martingale.
Furthermore, 
by Wald's lemma, i.e.,  since $(B_t-\xo)^2-t$ is a martingale,
\begin{equation}\label{waldn}
\E (B_{A_t\wedge n}-\xo)^2 = \E\lrpar{A_t\wedge n}. 
\end{equation}
It follows that
for any fixed $t\le t_0$, 
$\E (B_{A_t\wedge n}-\xo)^2 \le \E A_t<\infty$; hence
the variables
$B_{A_t\wedge n}$, $n\ge1$, are uniformly integrable, and
thus $B_{A_t\wedge n}\to B_{A_t}$ in $L^1$. 
If $0-\le s\le t\le t_0$, we thus obtain, by
letting $n\to\infty$ in \eqref{lmart}, 
$\E (B_{A_t}\mid \cF_{A_s}) =   B_{A_s}$ a.s.

Thus $X_t=B_{A_t}$, $0-\le t\le t_0$, is an integrable martingale; it is
square integrable by \refL{LWald}.
\end{proof}

Note that the converse to \refL{LWald}
does not always hold: we may have $\Var X_t<\infty$
also when $\E A_t=\infty$. For example, this happens in \refE{Epoint} if
$x_1\neq \xo$. We give a simple sufficient condition for $\E A_t<\infty$.

\begin{lemma}
  \label{LA}
Suppose that $\suppoo\nu\cap(-\infty,\xo]$ and $\suppoo\nu\cap[\xo,\infty)$
both are non-empty. Then $(X_t)_{t\ge0-}$,  is a bounded martingale with
$\E X_t=\xo$ and 
$\E A_t=\E (X_t-\xo)^2<\infty$ for every $t\ge0$.
\end{lemma}

\begin{proof}
\refL{LB} 
shows that
$\xsoo_-\le X_t=B_{A_t}\le \xsoo_+$, for any $t\ge0$,
so $(X_t)$ is uniformly bounded. 

For each $n$, $A_t\wedge n$ is a bounded stopping time and 
\eqref{waldn} holds.
Letting \ntoo,
we find 
$\E (X_t-\xo)^2=\E A_t$ by dominated
and monotone convergence, and thus
$\E A_t\le\max\set{(\xo-\xsoo_-)^2,(\xsoo_+-\xo)^2}<\infty$.
Finally, \refL{Lmart} shows that $X_t$ is a martingale.
\end{proof}

We have defined $A_t$ in \eqref{at} so that it is right-continuous.
The corresponding left-continuous process is
\begin{equation}  \label{atm}
  A_{t-}:=\inf\set{u\ge0:\gG_u\ge t};
\end{equation}
note that for $t>0$, $\atm=\lim_{s\upto t} A_s$, 
while $A_{0-}=0$ (as defined in Remark~\ref{RRdef}\ref{RRdisc} above), and 
$A_{t-}$ is a stopping time. It is possible that $\atm<A_t$, i.e.\ that
$A_t$  jumps; this corresponds to time intervals where $\gG_u$ is constant,
because $B_u$ moves in the complement of $\supp\nu$, so unless
$\supp\nu=\R$, it will a.s.\ happen for some $t$.
However, the next lemma shows that there is a.s. no jump for a fixed $t>0$.
(Equivalently, for a fixed $t>0$, there is a.s.\ 
at most one $u>0$ such that $\gG_u=t$.)

\begin{lemma}
  \label{LQ}
Let $t$ be fixed with $0<t<\infty$. 
Then a.s. $\atm=A_t$.
\end{lemma}

Note that the result fails for $t=0$, see \refL{L0}.

\begin{proof} 
Let $\eps>0$.
Since $\atm$ is a stopping time,
\refL{LT} shows that a.s.\ there exists 
a \nbh{} $U$ of $B_{A_{t-}}$  and
some $\gd>0$ such that
$L_{\atm+\eps}^{x}\ge L_{\atm}^x+\gd$ for $x\in U$. 
Further, since $A_{t-1/n}\to A_{t-}$ as \ntoo, \refL{Lsupp}
implies that a.s.\ $B_{A_{t-}}\in\supp\nu$ and thus $\nu(U)>0$; hence either
$\gG_\atm=\infty$ or
\begin{equation}
  \label{atme}
\gG_{\atm+\eps}\ge \gG_{\atm}+\gd\nu(U)>\gG_{\atm}.
\end{equation}

If $\gG_{\atm+\eps}<\infty$, then $\gG_u$ is continuous at $u=\atm$, 
see \refRR{RRdef}\ref{RRgGcont}, and it 
follows from \eqref{atm} that $\gG_\atm\ge t$ (and actually $\gG_\atm=t$);
thus \eqref{atme} yields $\gG_{\atm+\eps}> t$. This is trivially true also
when $\gG_{\atm+\eps}=\infty$.

Thus, a.s., $\gG_{\atm+\eps}>t$, which implies that
that $A_t\le\atm+\eps$. Since $\eps>0$ is arbitrary, and $\atm\le A_t$, the
result follows.
\end{proof}

When considering several speed measures $\nu_n$, we use $n$ as a superscript
to denote the corresponding $\gG_u^n$, $A_t^n$ and $X_t^n$; we use always
the same $B_t$.

If $S$ is a topological space (in our case $\R$ or an interval in $\R$),
we let $\cc(S)$ denote the space of continuous functions $S\to\R$ with
compact support, and $\ccp(S)$ the subset of such functions $S\to\ooo$.

\begin{lemma}
\label{Lconv}
Let $\nu, \nu_1,\nu_2,\dots$ be a sequence of measures on $\R$.
Assume either
\begin{romenumerate}
\item 
$\intR\gf\dd\nu_n \to\intR\gf\dd\nu$ as \ntoo{} for every $\gf\in\ccp(\R)$,
\end{romenumerate}
or, more generally,
\begin{romenumerateq}
\item 
there exists an interval $(a,b)$ with $-\infty\le a<\xo<b\le\infty$ such that
$\int\gf\dd\nu_n \to\int\gf\dd\nu$ as \ntoo{} for every $\gf\in\ccp(a,b)$,
and also for every $\gf\in\ccp(\R)$ such that $\gf(a)>0$ or $\gf(b)>0$.
\end{romenumerateq}

Then, for each $t>0$, $A^n_t\to A_t$ a.s., and thus, if
$\nu,\nu_1,\nu_2,\dots$ are non-zero,
$X^n_t\to X_t$ a.s.,
where
$X, X^1,\allowbreak X^2,\dots$ are the corresponding generalised
diffusions  constructed from the same Brownian motion. 
\end{lemma}

\begin{proof}
The local time $L_u^x\in\ccp(\R)$, as a function of $x$, for every $u\ge0$.
In (i) we thus have, for every $u\ge0$,
\begin{equation}
  \label{c2}
\gG_u^n
= \intR L_u^x\nu_n(dx)
\to \intR L_u^x\nu(dx) = \gG_u.
\end{equation}

In (ii), let $H=H_{\set{a,b}}$ be the hitting time of \set{a,b}.
If $u<H$, then the support of $L_u^x$ is contained in $(a,b)$, so
$L_u^x\in\ccp(a,b)$ and $\gG_u^n\to\gG_u$ as in \eqref{c2}.

If $u>H$, then a.s.\ $L_u^{B_H}>0$ by \refL{LT}. 
Since $\gf(x):= L_u^x$ is continuous
and $B_H=a$ or $B_H=b$, the assumption shows that \eqref{c2} holds in this
case too.

Hence, in both (i) and (ii), $\gG_u^n\to\gG_u$ for all $u\ge0$ except
possibly when $u=H$.

Let $s=A_t$. By \eqref{at}, $\gG_u>t$ for $u>s$ and $\gG_u\le t$ for $u<s$.
Further, if $\gG_u=t$ for some $u<s$, then
$\atm\le u<s=A_t$,
which has probability  0 by \refL{LQ}.
Consequently, a.s.\ $\gG_u<t$ if $u<s$.

Assume for simplicity that $s=A_t<\infty$. (The case $A_t=\infty$ is
similar.)
If $\eps>0$ and $s+\eps\neq H$, then thus $\gG^n_{s+\eps}\to\gG_{s+\eps}>t$.
Hence, for sufficiently large $n$, $\gG_{s+\eps}^n>t$ and thus $A_t^n\le
s+\eps$. Since $\eps>0$ is almost arbitrary, it follows that 
$\limsup_\ntoo A_t^n\le s$ a.s.
Similarly, considering $s-\eps\neq H$, it follows that
$\liminf_\ntoo A_t^n\ge s$ a.s.
Consequently, $A_t^n\to A_t$ a.s.\ as \ntoo.

Finally, if $A_t<\infty$, then  $X_t^n=B_{A_t^n}\to B_{A_t}= X_t$ a.s.\
follows, since $B_u$ is continuous. 
\end{proof}

\begin{lemma}\label{LC}
Let $a>b>\xo$ and $\gd>0$.
Then there exists a constant $C=C(a,b,\xo,\gd)$ such that 
if $\P(X_1\ge a)\ge\gd$, then $\nu[\xo,b]\le C$.  
\end{lemma}

\begin{proof}
Assume for convenience $\xo=0$.
By replacing $a$ by $(a+b)/2$, we may also assume that $\P(X_1>a)\ge\gd$.

Let $H=H_a=\inf\set{u:B_u=a}$. By definition, $X_1=B_{A_1}$, so if $X_1>a$, then
$H<A_1$ and thus $\gG_H\le 1$. Consequently, $\P(\gG_H\le 1)\ge\gd$.

The local time $L_H^x$ is a continuous function of $x\in\R$, and it is 
a.s.\ strictly positive on $[0,a)$
by Ray's theorem \cite[Thm 6.38]{MPeres} (a consequence of the Ray--Knight
theorem which gives its distribution). Hence, 
$Y=\inf_{x\in[0,b]} L_H^x>0$ a.s., and thus there is a constant $c>0$ such
that $\P(Y<c)<\gd$. 

Hence, with positive probability $\gG_H\le1$ and $Y\ge c$. However, then
\begin{equation*}
1\ge  \gG_H=\intR L_H^x\nu(dx)
\ge\int_0^b L_H^x\nu(dx)\ge Y\nu[0,b] \ge c\nu[0,b],
\end{equation*}
so $\nu[0,b]\le 1/c$.
\end{proof}

\begin{lemma}\label{Llower}
For every $K>0$ there exists $\kk=\kk(K)>0$ such that  
if $\nu$ is a speed measure such that $\E |X_1|\le K$, 
and further $\suppoo\nu\cap(-\infty,\xo]$ and $\suppoo\nu\cap[\xo,\infty)$
both are non-empty,
then $\nu[\xo-2K,\xo+2K]\ge\kk$. 
\end{lemma}

\begin{proof}
We may assume that $\xo=0$.
By  \refL{LB}, $(B_{u\wedge A_1})_{u\ge0}$ is a bounded, and thus uniformly
integrable, martingale, closed by $B_{A_1}=X_1$.   

Let $\HH=H_{\set{\pm2K}}$ be the hitting time of
${\pm2K}$. Then 
\begin{equation}
\P(\gG_{\HH}\le1)=
 \P(\HH\le A_1)\le \P\bigpar{\sup_u|B_{u\wedge A_1}|\ge 2K}  
\le \frac{\E|B_{A_1}|}{2K}
\le\frac{K}{2K}=\frac12.
\end{equation}

Let $Y=\max_x L^x_\HH$; this is a finite random variable so there
exists $c>0$ such that $\P(Y>c)<1/2$. (Note that $Y$ and  $c$ depend on
$K$ but not on $\nu$.) With positive probability we thus have
both $\gG_\HH>1$ and $Y\le c$. This implies, since $L^x_\HH=0$ when $|x|>2K$,
\begin{equation*}
  \begin{split}
  1&<
\gG_\HH=\int_{-2K}^{2K} L^x_\HH\, \nu(dx) 
\le \int_{-2K}^{2K} Y \,\nu(dx) 
\le c\nu[-2K,2K],
  \end{split}
\end{equation*}
and the result follows with $\kk=c^{-1}$.
\end{proof}

\section{The discrete case} 
\label{discrete}

In this section we treat the inverse problem in a discrete setting. 
We fix points $y_0 < y_1 < y_2 <\dots<y_{n+1}$ and
consider 
discrete speed measures $\nu =\sum_{i=0}^{n+1}b_i\delta_{y_i}$,
where $\delta_{a}$ is a unit point mass at the point $a$ and $b_i$ takes
values in $[0,\infty]$. We assume that $b_0=b_{n+1}=\infty$.
We also fix a starting point $\xo\in(y_0,y_{n+1})$. 
(We could for simplicity assume that
$\xo=y_{i_0}$ for some $i_0\in\{1,2,\dots,n\}$, but that is not
necessary.) 

Given such a speed measure $\nu$ and $\xo$, we construct a generalised
diffusion $X$ as 
described in Section~\ref{construction} above.
By \refL{Lsupp}, the process $(X_t)_{t\ge0}$ only takes values in the set 
$\{y_i\}_{i=0}^{n+1}$. Moreover, since $b_0=b_{n+1}=\infty$, the states 
$y_0$ and $y_{n+1}$ are absorbing, so $X$ is bounded, and it
follows from \refL{Lmart} (or \refT{Tekstrom-hobson}) that $X$ is a
martingale; in particular $\E X_t=\xo$.
Let $p_i=\P(X_1=y_i)$ be the probability that $X$ at time 1 is in state $y_i$.
Then $0\leq p_i\leq 1$, $\sum_{i=0}^{n+1} p_i=1$, and 
we also have $\sum_{i=0}^{n+1} y_ip_i=\E X_1=\xo$.

This defines a mapping $G$ from the set of speed measures above to the set of 
distributions with mean $\xo$.
More precisely, we write $G(b_1,\dots,b_n):=(p_0,\dots,p_{n+1})$, and note that
$G:\bn\to \pin$, where $\bn:=[0,\infty]^n$ and
\[\pin:= \Bigset{\pi=(\pi_0,\dots,\pi_{n+1})\in [0,1]^{n+2}:
\sum_{i=0}^{n+1} \pi_i=1\mbox{ and } \sum_{i=0}^{n+1} y_i\pi_i=\xo}
.\]

\begin{lemma}
\label{continuity}
The function $G:\bn\to\pin$ is continuous.
\end{lemma}

\begin{proof}
This is a direct consequence of Lemma~\ref{Lconv}.
(Note that the possibility that one or several $b_i=\infty$ is no
problem when we verify condition (i).)
\end{proof}
 
We will use algebraic topology to show that $G$ is surjective; see e.g.\ 
\cite[Chapter IV]{Bredon} for standard definitions and results used below.
We begin by studying the sets $\bn$ and $\pin$.
The set $\bn$ is homeomorphic to the unit cube $[0,1]^n$ with
the boundary 
$\partial \bn:=
[0,\infty]^n\setminus(0,\infty)^n
$
corresponding to the boundary $\partial[0,1]^n=[0,1]^n\setminus(0,1)^n$.
We write
\begin{equation}\label{bdB}
\partial \bn=
 \bigcup_{j=1}^n \bigpar{\partial_{j0}\bn \cup \partial_{j\infty}\bn}
\end{equation}
where
$\partial_{js}\bn:=\set{(b_i)_1^n\in\bn:b_i=s}$.

The set
$\pin$ is the intersection $\gdni\cap\mxo$ of the simplex 
\begin{align*}
\gdni&:=
\Bigset{(\pi_i)_{0}^{n+1}:\pi_i\ge0\text{ and } \sum_{i=0}^{n+1}\pi_i=1}
\intertext{and the hyperplane}
\mxo&:=\Bigset{(\pi_i)_{0}^{n+1}:\sum_{i=0}^{n+1}y_i\pi_i=\xo}
\end{align*}
in $\R^{n+2}$.
Further, $\gD_{n+1}$ lies in the hyperplane 
$L:=\bigset{(\pi_i)_{0}^{n+1}:\sum_{i=0}^{n+1}\pi_i=1}$.
Thus $\pin$ is a compact convex set in the $n$-dimensional plane
$L\cap\mxo$,
which can be identified with $\R^n$.
Since $y_0<\xo<y_{n+1}$, $\mxo$ contains an interior point of $\gdni$, so
$\pin=\gdni\cap\mxo$ has a non-empty relative interior in
$L\cap\mxo$. 
Thus
$\pin$ is an $n$-dimensional compact convex set in $L\cap\mxo$
and its boundary is 
\begin{align*}
  \partial\pin=\pin\cap\partial\gdni
=\bigcup_{j=0}^{n+1}\partial_j\pin,
\end{align*}
where $\partial_j\pin:=\bigset{(\pi_i)_{0}^{n+1}\in\Pi^n:\pi_j=0}$. 

Consequently, both $\bn$ and $\pin$ are homeomorphic to compact convex sets
in $\R^n$ with non-empty interiors. Every such set is homeomorphic to the
unit ball $D^n:=\set{x\in\R^n:|x|\le1}$ via a homeomorphism mapping the
boundary onto the boundary $\bd D^n=\sni$.
Thus there are homeomorphisms $(\bn,\bd\bn)\approx(D^n,\sni)$ and
$(\pin,\bd\pin)\approx(D^n,\sni)$.

\begin{lemma}
\label{Lbdy}
The function $G$ maps $\bd\bn$ into $\bd\pin$, and thus
$G:(\bn,\bd\bn)\to(\pin,\bd\pin)$.
\end{lemma}

\begin{proof}
  If $\bb\in\bd_{j0}\bn$, then $y_j\notin\supp\nu$, so $X_1\neq y_j$ a.s.\
  by \refL{Lsupp} and thus $\pi_j=0$,  so $\ppi\in\bd_{j}\pin$.

If $\bb\in\bd_{j\infty}\bn$, then $y_j\in\suppoo\nu$.
If further $y_j\le\xo$, then \refL{LH} implies that $X_1\ge y_j$ a.s., and
thus $\pi_i=0$ for $0\le i<j$ and, e.g., 
$\ppi\in\bd_{j-1}\pin$.
Similarly, if $y_j\ge\xo$, then $\pi_i=0$ for $j<i\le n+1$
and $\ppi\in\bd_{j+1}\pin$.
\end{proof}

The homeomorphisms above induce isomorphisms of the relative
homology groups
$H_n(\bn,\bd\bn)\approx H_n(\pin,\bd\pin)\approx H_n(D^n,\sni)\approx \Z$.
The mapping degree of the function $G:(\bn,\bd\bn)\to(\pin,\bd\pin)$ can thus
be defined as the integer $\deg(G)$ such that the  homomorphism 
$G_*:H_n(\bn,\bd\bn)\to H_n(\pin,\bd\pin)$ corresponds to multiplication by
$\deg(G)$ on $\Z$. More precisely, this defines the mapping degree up to
sign; the sign depends on the orientation of the spaces, but we have no
reason to care about the orientations so we ignore them and the sign of
$\deg(G)$.

\begin{lemma}\label{Ldeg}
For any $n\ge1$, any $y_0<y_1<\dots<y_{n+1}$, and $\xo\in(y_0,y_{n+1})$,
  $\deg(G)=\pm1$.
\end{lemma}

\begin{proof}
We use induction on the dimension $n$. We sometimes write $G=G_n$ for clarity.
For the induction step, we assume $n\ge2$.
%
The long exact homology sequence yields
the commutative diagram
\begin{equation*}
\minCDarrowwidth20pt
  \begin{CD}
0=H_n(\bn)@>>> H_n(\bn,\bd\bn)@>\bd>> H_{n-1}(\bd\bn)@>>> H_{n-1}(\bn)=0  \\
   @.          @VVG_*V                 @VVG_*V             @.             \\
0=H_n(\pin)@>>> H_n(\pin,\bd\pin)@>\bd>> H_{n-1}(\bd\pin)@>>> H_{n-1}(\pin)=0	
  \end{CD}
\end{equation*}
where the rows are exact; thus the connecting homomorphisms $\bd$ are 
isomorphisms, and the degree of $G:(\bn,\bd\bn)\to(\pin,\bd\pin)$ equals the
degree of the restriction $G:\bd\bn\to\bd\pin$.

Assume $\xo< y_n$. (Otherwise $\xo\ge y_n>y_1$, and we may argue similarly
using $\bd_{1\infty}\bn$ and $\bd_{0}\pin$.)
We single out the faces $\bd_{n\infty}\bn$ and $\bd_{n+1}\pin$ of the
boundaries and define
$\bd_*\bn:=\bigcup_{i=1}^n\bd_{i0}\bn\cup\bigcup_{i=1}^{n-1}\bd_{i\infty}\bn$
  and
$\bd_*\pin:=\bigcup_{i=0}^n\bd_{i}\pin$.
By the proof of \refL{Lbdy}, $G:\bd_{n\infty}\bn\to\bd_{n+1}\pin$  and
$G:\bd_{*}\bn\to\bd_{*}\pin$. 

We claim that the degree of $G:\bd\bn\to\bd\pin$
equals the degree of 
$G:
(\bd_{n\infty}\bn,\bd_{n\infty}\bn\cap \bd_{*}\bn)
\to(\bd_{n+1}\pin,\bd_{n+1}\pin\cap \bd_{*}\pin)$.
Using homeomorphisms $\bd\bn\approx\sni$
and
$\bd\pin\approx\sni$ that map the faces
$\bd_{n\infty}\bn$ and $\bd_{n+1}\pin$ onto the upper hemisphere $\sni_+$,
this is an instance of the general fact that if $F:\sni\to\sni$ is
continuous and maps $\sni_\pm\to\sni_\pm$, then the degree of 
$F:(\sni_+,\sni_+\cap\sni_-)\to(\sni_+,\sni_+\cap\sni_-)$ equals
the degree of $F:\sni\to\sni$.

If $\bb\in\bd_{n\infty}\bn$, then $\nu$ has infinite point masses at both
$y_n$ and $y_{n+1}$, with $\xo<y_n<y_{n+1}$. By \refL{LH}, we can ignore
$y_{n+1}$ and we obtain the same generalized diffusion $X$ as with the speed
measure $\sum_{i=0}^n b_i\gd_{y_i}$. We can thus identify $\bd_{n\infty}\bn$
with $\bni$ (based on the points $(y_i)_0^n$). Furthermore, there is an
obvious identification $\bd_{n+1}\pin=\pini$, and with these identifications,
$G:\bd_{n\infty}\bn\to\bd_{n+1}\pin$ corresponds to
$G_{n-1}:\bni\to\pini$.
Moreover, the various boundaries correspond so that we have the commutative
diagram
\begin{equation*}
\minCDarrowwidth10pt
  \begin{CD}
H_{n-1}(\bd_{n\infty}\bn,\bd_{n\infty}\bn\cap \bd_{*}\bn)
 @= H_{n-1}(\bni,\bd\bni)
\\
   @VVG_{n\,*}V                 @VVG_{n-1\,*}V     \\
H_{n-1}(\bd_{n+1}\pin,\bd_{n+1}\pin\cap \bd_{*}\pin)
 @= H_{n-1}(\pini,\bd\pini)
  \end{CD}
\end{equation*}
where the rows are the isomorphisms given by these identifications.
Hence 
the degree of 
$G_n:
(\bd_{n\infty}\bn,\bd_{n\infty}\bn\cap \bd_{*}\bn)
\to(\bd_{n+1}\pin,\bd_{n+1}\pin\cap \bd_{*}\pin)$
equals $\deg(G_{n-1})$.

Combining this with the equalities above, we see that 
$\deg(G_{n})=\deg(G_{n-1})$, which completes the induction step.

It remains to treat the initial case $n=1$. In this case $B^1$ and $\Pi^1$
are intervals, and can be parametrized by $b_1$ and $p_1$. It is easy to see
that the mapping $G:b_1\mapsto p_1$ is strictly increasing, and thus 
a homeomorphism $B^1\to\Pi^1$; hence
$G_*:H_1(B^1,\bd B^1)\to H_1(\Pi^1,\bd\Pi^1)$ is an isomorphism, so
$\deg(G)=\pm1$. 
Alternatively, we may use the first commutative diagram above also in the
case $n=1$, replacing the
homology groups $H_{n-1}=H_0$ by the reduced homology groups $\tH_0$.
The sets $\bd B^1$ and $\bd\Pi^1$ contain exactly two elements each, and it
is easy to see that $G:\bd B^1\to\bd\Pi^1$ is a bijection and thus 
$G_*:\tH_0(\bd B^1)\to\tH_0(\bd\Pi^1)$ is an isomorphism.
\end{proof}

\begin{theorem}
\label{main}
The function $G$ is surjective,
for any $n\ge1$ and any $y_0<y_1<\dots<y_{n+1}$ and $\xo\in(y_0,y_{n+1})$.
Consequently, the discrete inverse problem has a solution.
\end{theorem}

\begin{proof}
An immediate consequence of  \refL{Ldeg}, since a function 
$(\bn,\bd\bn)\to (\pin,\bd\pin)$ 
(or, equivalently, 
$(D^n,\sni)\to(D^n,\sni)$)
that is not
surjective has mapping degree 0. 
\end{proof}

\section{The general case}\label{general}

In this section we study the inverse problem for arbitrary distributions
on the real axis. To do this, we approximate the given distribution with a 
sequence of discrete distributions. For each discrete distribution we can find
a discrete speed measure that solves the inverse problem 
according to Section~\ref{discrete}. We then show that
the sequence of discrete speed measures has a convergent subsequence,
and that the limit solves the inverse problem.
We begin with a lemma giving the approximation that we shall use.

\begin{lemma}\label{Lmy}
  Let $\mu$ be a probability measure on $\R$ with finite mean 
$\bmu=\intR x\,\mu(dx)$. Then 
  there exists a sequence $\mu_n$, $n\ge1$, of probability measures with finite
  supports such that, as \ntoo,
  \begin{romenumerate}
  \item \label{LLw}
$\mu_n\to\mu$ weakly;
  \item \label{LLinf}
$\inf\supp\mu_n\to\inf\supp\mu$;
  \item \label{LLsup}
$\sup\supp\mu_n\to\sup\supp\mu$;
  \item  \label{LLmean}
each $\mu_n$ has the same mean $\bmu$ as $\mu$.
\item \label{LL1} 
$\intR|x|\,d\mu_n(x)\to\intR|x|\,d\mu(x)$.
  \end{romenumerate}
If further $\mu$ has finite variance $\Var(\mu)=\intR(x-\bmu)^2\mu(dx)$, then
$\mu_n$ can be chosen such that
\begin{romenumerateq}
\item  \label{LLvar}
$\Var(\mu_n)\to\Var(\mu)$.
\end{romenumerateq}
\end{lemma}

\begin{proof}
Let $Y$ be a random variable with distribution $\mu$. First, truncate $Y$ at
$\pm n$ by defining
\begin{equation*}
  Y_n':=(Y\wedge n)\vee(-n).
\end{equation*}
Then, letting all limits in this proof be for \ntoo, 
\begin{equation*}
  \E|Y_n'-Y|\le\E(|Y|;|Y|>n)
\to0.
\end{equation*}
Next, discretize by defining
\begin{equation*}
  Y_n'':=\frac1n\floor{nY_n'}.
\end{equation*}
Clearly, $|Y_n''-Y_n'|<1/n$. It follows that
\begin{equation*}
  |\E Y_n''-\E Y| \le \E|Y_n''-Y|
\le\frac1n+\E|Y_n'-Y|\to0.
\end{equation*}
Finally, we adjust the mean by defining 
\begin{equation}\label{yn}
Y_n:=Y_n''-\E(Y_n''-Y).  
\end{equation}
Thus $\E Y_n=\E Y=\bmu$.

Let $\mu_n=\cL(Y_n)$, the distribution of $Y_n$. Then $\mu_n$ has
finite support and (iv) holds by the construction.
Furthermore, by \eqref{yn},
  \begin{equation}\label{ynk}
\E|Y_n-Y|	
\le
2\E|Y_n''-Y| \to0	,
  \end{equation}
which implies $Y_n\dto Y$ and thus  (i).
From \eqref{ynk} we also have $\E|Y_n|\to\E|Y|$, which is \ref{LL1}.

If $\inf\supp\mu=-\infty$, then (i) implies that $\inf\supp\mu_n\to-\infty$,
so (ii) holds.

Suppose now that $\inf\supp\mu=a>-\infty$. If $n>|a|$, then
$\inf\supp\cL(Y_n')=a$, and it follows that
\begin{equation*}
 |\inf\supp\mu_n-a|\le\frac1n+|\E Y_n''-\E Y| \to0,
\end{equation*}
which shows that (ii) hold in this case too.

The proof of (iii) is similar, mutatis mutandis.

If $\mu$ has finite variance, then $\E Y^2<\infty$, and
$\E|Y_n'|^2 = \E(|Y|\wedge n)^2 \to \E Y^2$.
Taking square roots we find $\normq{Y_n'}\to\normq{Y}$.
Minkowski's inequality yields
\begin{equation*}
  \bigabs{\normq{Y_n}-\normq{Y_n'}}
\le \normq{Y_n-Y_n'}\le\frac1n+\bigabs{\E(Y_n''-Y)}
\to0.
\end{equation*}
Consequently, $\normq{Y_n}\to\normq{Y}$, and thus $\E Y_n^2\to\E Y^2$.
Since $\E Y_n=\E Y$, this implies $\Var(Y_n)\to\Var(Y)$, which shows (v).
\end{proof}

\begin{proof}[Proof of \refT{cont}]
We may for simplicity assume that $\xo=\bmu=0$.

Let $a_-=\inf\supp\mu\ge-\infty$ and $a_+=\sup\supp\mu\le+\infty$.
Since $\bmu=0$, we have $a_-\le0\le a_+$. 
Moreover, if $a_-=0$ or $a_+=0$, 
then necessarily $\mu=\gd_0$. In this case we may simply take
$\nu$ as an infinite point mass at 0; then $A_t=0$ and $X_t=0$ for all
$t\ge0$ a.s., see
\refE{Epoint}. In the sequel we thus assume $-\infty\le a_-<0$ and
$0<a_+\le \infty$.
 
Let $\mu_n$ be a sequence of distributions satisfying \ref{LLw}--\ref{LL1} in
\refL{Lmy}.
The distributions $\mu_n$ have finite supports, and
thus Theorem~\ref{main} shows that there exist speed measures $\nu_n$ 
so that the corresponding generalised diffusion $X^n$ has distribution
$\mu_n$ at time 1.

If $0<b<a_+$, choose $a\in(b,a_+)$. Then $a>b>0$ and $a<\sup\supp\mu$, so
$\mu(a,\infty)>0$.
Since $\mu_n\to\mu$, $\liminf_\ntoo\mu_n(a,\infty)\ge\mu(a,\infty)$, so for
all large $n$, $\mu_n(a,\infty)>\frac12\mu(a,\infty)>0$.
\refL{LC} applies and implies that $\nu_n[0,b]\le C=C(b)$ for all large $n$,
i.e., $\limsup_\ntoo\nu_n[0,b]<\infty$ for every $b<a_+$.

By a symmetric argument, we also have $\limsup_\ntoo\nu_n[b,0]<\infty$ for
every $b>a_-$.

Choose sequences $b_-^m\downto a_-$ and $b_+^m\upto a_+$.
On each interval $\bbm$, the measures $\nu_n$ are, as we just have shown,
uniformly bounded if we exclude a finite number of small $n$.
Since $\bbm$ is compact, we may thus choose a subsequence of $\nu_n$ such
that the restrictions to $\bbm$ converge to some measure $\hat\nu_m$.
By a diagonal procedure, we can do this simultaneously for all $m$, which
provides a subsequence of $\nu_n$ such that (along the subsequence)
$\nu_n\to\hat\nu_m$ on $\bbm$ for every $m$. In the sequel we consider only
this subsequence.

It follows that $\gL(f)=\lim_\ntoo\intR f\dd\nu_n$ exists and is finite for
every $f\in\ccp(a_-,a_+)$. Clearly, $\gL$ is a positive linear functional on
$\ccp\aax$, so by the Riesz representation theorem there exists a Borel
measure $\tnu$ on $\aax$ with $\gL f=\intR f\dd\tnu$.
Thus $\intR f\dd\nu_n\to\int f\dd\tnu$ for all $f\in\cc\aax$.

We define $\nu$ by adding infinite point masses at $a_-$ and $a_+$, if these
are finite:
$\nu=\tnu+\infty\cdot\gd_{a_-}+\infty\cdot\gd_{a_+}$
(where $\gd_{\pm\infty}=0$).

If $a_-$ is finite and $f\in\ccp(\R)$ with $f(a_-)>0$, then 
$\int f\dd\nu\ge f(a_-)\nu\set{a_-}=\infty$.
Furthermore, $a_-^n:=\inf\supp\mu_n\to a_-$, so $f(a^n_-)>0$ for all large
$n$. The construction of $\nu_n$ in \refT{main} gives $\nu_n$ an infinite
point mass at $a_-^n$, so $\int f\dd\nu_n=\infty$ for all large $n$. Thus
$\int f\dd\nu_n\to\int f\dd\nu=\infty$ as \ntoo.
Similarly, $\int f\dd\nu_n\to\int f\dd\nu=\infty$ as \ntoo{} if $f(a_+)>0$.

The assumptions of \refL{Lconv}(ii) are satisfied, and thus a.s.\
$A^n_1\to A_1$ and, if $\nu\neq0$, $X_1^n\to X_1$ as \ntoo. 
In particular, then $X_1^n\to X_1$ in distribution,
and since $X_1^n$ has distribution $\mu_n$ and $\mu_n\to\mu$, it follows
that the distribution of $X_1$ is $\mu$.

It remains to verify that the  measure $\nu$ is non-zero.
If $a_-$ or $a_+$ is finite, this is clear since $\nu$ by construction has a
point mass there. 

If
$a_\pm=\pm\infty$, we use \refL{Lmy}\ref{LL1} which yields
$\E|X_1^n|=\intR|x|\,d\mu_n(x)\to\intR|x|\,d\mu(x)<\infty$.
Let $K:=\sup_n\E|X_1^n|<\infty$.
By \refL{Llower}, there exists $\kk>0$ such that $\nu_n[-2K,2K]\ge\kk$ for
every $n$.
Let $f\in \ccp(-\infty,\infty)$ with $f=1$ on $[-2K,2K]$.
As shown above, $\intR f\,d\nu_n\to\intR f\,d\tnu=\intR f\,d\nu$ as \ntoo.
Since $\intR f\,d\nu_n\ge\nu_n[-2K,2K]\ge\kk$, this implies $\intR
f\,d\nu\ge\kk>0$, and thus $\nu\neq0$.

This proves the existence of a non-zero speed measure $\nu$ such that $X_1$
has the desired distribution $\mu$.
We next prove that $X_t$ is a  martingale.

We have shown that $X^n_1\to X_1$ a.s.\
and, by \refL{Lmy}\ref{LL1}, $\E|X_1^n|\to\E|X_1|$. 
This implies $\E|X_1^n-X_1|\to0$, i.e.\ $X_1^n\to X_1$ in $L^1$
(see e.g.\ \cite[Proposition 4.12]{Kallenberg}).
For each $n$, $(B_{A_1^n\wedge u})_{u\ge0}$ is 
by \refL{LB}
a bounded martingale with limit $B_{A_1^n}=X_1^n$, and thus
$B_{A_1^n\wedge u}=\E(X_1^n\mid\cF_u)$ for every $u$.
As \ntoo, a.s.\ $A_1^n\to A_1$ and thus $B_{A_1^n\wedge u}\to B_{A_1\wedge u}$.
Further, we have just shown $X_1^n\to X_1$ in $L^1$, and this implies 
$B_{A_1^n\wedge u}=\E(X_1^n\mid\cF_u)\to \E(X_1\mid\cF_u)$ in $L^1$. 
The two limit results both hold in probability, so the limits must coincide:
$B_{A_1\wedge u}= \E(X_1\mid\cF_u)$.

This proves that $(B_{A_1\wedge u})_{u\ge0}$ is a uniformly integrable
martingale. Consequently, for any $(\cF_t)$-stopping time $\tau$,
$B_{A_1\wedge \tau}= \E(X_1\mid\cF_\tau)$.
In particular, for $0-\le t\le 1$,
$X_t=B_{A_t}= \E(X_1\mid\cF_{A_t})$, which proves that $(X_t)_{t\le 1}$ is a
martingale. 
This completes the main part of the proof, and we turn to \ref{cont-X_0} and
\ref{cont-var}.

By \refL{L0}, $X_0=\bmu$ if and only if $\bmu\in\supp\nu$.
If $\bmu\notin\supp\nu$, then there exists an open set $U$ with $\bmu\in U$
and $\nu(U)=0$. By \refL{Lsupp}, $X_1\notin U$ a.s., so
$\bmu\notin\supp\mu$.
On the other hand, if $\bmu\in\supp\nu$ and $U$ is any open set containing
$\bmu$, then $\nu(U)>0$. It is easy to see that there is a positive
probability that $B_t$ will remain in $U$ until $\gG_t>1$, and thus $X_1\in
U$. Hence $\mu(U)=\P(X_1\in U)>0$ for any such $U$, so $\bmu\in\supp\mu$.

If $\mu$ has finite variance, 
we may by \refL{Lmy}(v) assume that $\Var\mu_n\to\Var\mu$, and thus
$\sup_n\Var\mu_n<\infty$. 
\refL{LA} applies to every $\nu_n$ and yields
\begin{equation*}
  \E A_1^n = \E (X_1^n)^2=\Var(\mu_n).
\end{equation*}
Consequently, by Fatou's lemma,
\begin{equation}\label{emu}
  \E A_1=\E\lim_\ntoo A_1^n 
\le \liminf_\ntoo\E A_1^n
= \liminf_\ntoo\Var(\mu_n)
=\Var(\mu)<\infty.
\end{equation}
\refL{LWald} shows that $\E A_1=\Var(\mu)$ in this case (so
there is equality in \eqref{emu}), and also if $\Var(\mu)=\infty$.
\end{proof}

\newcommand{\BBB}{B^{(3)}}
\begin{example}
As an illustration of Theorem~\ref{cont}, let $Y_t=|\BBB_t|^{-1}$, 
where $\BBB$ is a
3-dimensional Brownian motion with $|\BBB_0| = \yo^{-1}$ for some $\yo> 0$.
It is well-known that $Y$ is a local martingale bounded from below, hence a supermartingale,
but it is not a true martingale, compare \refT{Tmartingale} below.
Moreover, $Y$ can be represented as the solution of
\begin{equation}\label{x1}
  \begin{cases}
dY_t=Y^2_t\,dW_t, & t>0,\\
Y_0=\yo	
  \end{cases}
\end{equation}
for some standard Brownian motion $W$.
Note that $Y$ is a diffusion by \eqref{x1}, and by \refE{Ediffusion} 
(together with a stopping argument), $Y$ is
the generalized diffusion with speed measure $x^{-4}dx$, $x>0$.
(Note that $Y_t$ tends to 0 as $\ttoo$, but never reaches 0.)
Being the reciprocal of a Brownian motion, 
$Y$ has an explicit density, compare \cite[Example  2.2.2]{CH}.
The density of $Y_1$ decays like $Cx^{-4}$
for large $x$, so only moments of order strictly less than three exist finitely.
Moreover, the expected value of $Y_1$ is strictly smaller than the starting
value $\yo$. 

Theorem~\ref{cont} provides the existence of a time-homogeneous
generalised diffusion  
$X$ which is a true martingale such that $X_1$ and $Y_1$
have the same distribution, and $X_0=\E Y_1<y_0$. Consequently, there are in this case two different 
speed measures that give rise to the same distribution at time 1. However, only one
of the corresponding processes is a martingale, and they have different starting points.
\end{example}

\begin{example}  
\label{X2}
A related example is when $Y$ is the solution of 
\begin{equation*}
  \begin{cases}
dY_t=(1+Y^2_t)\,dW_t, & t>0,\\
Y_0=0.
  \end{cases}
\end{equation*}
By \refE{Ediffusion}, the diffusion $Y$ is the generalized
diffusion with speed measure
$(1+x^2)^{-2}dx$.  
Again, $Y$ is a local martingale, but not a martingale, see 
Theorems \ref{Tekstrom-hobson} and \ref{Tmartingale} below. 

Define $h(x,t)=e^{Mt}g(x)$, where $g$ is a smooth positive function satisfying
$g(x)=|x|$ for $|x|\geq 1$ and $g(x)\geq |x|$ everywhere, and $M$ is a positive
constant so that $h_t\geq \frac{1}{2}(1+x^2)^2h_{xx}$.
Let $u_N(x,t)=\E_x \vert Y_t\vert\wedge N$, where the index indicates that $Y_0=x$.
By a maximum principle argument, compare \cite{ET}, 
$u_N(x,t)\leq h(x,t)$ independently of $N$. Consequently, by monotone convergence, 
$\E_0 |Y_t|\leq h(0,t)<\infty$.

Theorem~\ref{cont} thus applies and provides a
generalised diffusion $X_t$ which is a martingale started at 0 such that $X_1$ has 
the same distribution as $Y_1$. Consequently, we have in this case two different 
generalised diffusions with the same starting point
that give rise to the same distribution at time 1. However, only one
of these processes is a martingale.
\end{example}

\section{An application to Mathematical Finance}
\label{finance}  

In this section we study an inverse problem in Mathematical Finance. Let $X$
be a non-negative 
martingale with $X_0=\xo$, and consider the expected values
\begin{equation}
\label{C}
C(K,T):=\E (X_T-K)^+.
\end{equation}
Here $X_t$ has the interpretation as a price process of a stock,
and the expected value $C(K,T)$ is the price of a call option with strike $K$ and maturity $T$
(for the sake of simplicity, we assume that interest rates are zero). The function $C$ is 
non-increasing and convex as a function of $K$, and it satisfies
$C(0,T)=\xo$, $C(\infty,T)=0$ and 
$C(K,T)\ge (\xo-T)^+$.

In Mathematical Finance, the corresponding inverse problem is of great interest.
Option prices $C(K,T)$ are observable in the market, at least for a large 
collection of strikes $K$ and maturities $T$, whereas the stock price model
is not.
Under some regularity conditions (often neglected in the literature),
Dupire \cite{D} determines a local volatility model 
\[dX_t=\sigma(X_t,t)\,dW_t\]
such that \eqref{C} holds for all $K$ and $T$. To do this, naturally
one needs call option prices $C(K,T)$ given
for all strikes $K\geq 0$ and all maturities $T>0$.

The assumption that call option prices are known for all maturities $T$ is 
often unnatural in applications - indeed, there is typically only one maturity 
a month. We therefore consider the situation where $C(K,T)$ is given for all
strikes $K$ but for \emph{one} fixed maturity $T>0$. A natural question is then
whether there exists a time-homogeneous diffusion process
\[dX_t=\sigma(X_t)\,dW_t\]
such that \eqref{C} holds for all strikes. We have the following result.

\begin{theorem}
\label{main-finance}
Let $c:[0,\infty)\to[0,\infty)$ be a convex and non-increasing function satisfying
$c(0)=x_0$, $c(\infty)=0$ and $c(K)\geq (\xo-K)^+$,
and let $T>0$ be fixed. 
Then there exists a 
time-homogeneous (generalised) diffusion such that $X$ is a martingale with
$X_{0-}=x_0$, and such that $c(K)=\E(X_T-K)^+$.
\end{theorem}

\begin{proof}
We construct a probability distribution $\mu$ on $[0,\infty)$ as follows.
On $(0,\infty)$, we let $\mu$ be defined as the measure given by the 
second derivative of the convex function $c$. Since $c(\infty)=0$, this
measure has mass $-c'(0)$, where $c'(0)$ is the (right) derivative 
of $c$ at 0. Moreover, we let a point mass $\ep=1+c'(0)$ be located at 0; 
note that the assumptions 
$c(0)=x_0$ and $c(K)\geq (\xo-K)^+\ge \xo-K$ imply that $c'(0)\ge-1$ so
$\ep\ge0$. 

Integration by parts shows that the expected value of 
the distribution $\mu$ is given by 
\[\int_0^\infty x\mu(dx) = 
\int_0^\infty \mu(x,\infty)\,dx = 
-\int_0^\infty c'(x)\,dx = c(0)=\xo.\]
Now, by Theorem~\ref{cont} (and an obvious scaling to general $T$)
there exists a time-homogeneous 
generalised diffusion $X$ with $X_{0-}=\xo$ such that $X$ is a martingale and
$X_T$ has distribution $\mu$. Finally, integration by parts gives
\[\E(X_T-K)^+=\int_K^\infty (x-K)\mu(dx)= 
\int_K^\infty \mu(x,\infty)\,dx = 
c(K),\]
thus finishing the proof.
\end{proof}

\section{Martingality of generalised diffusions}
\label{martingality}
The speed measure constructed in the proof of \refT{cont} is such that the 
generalized diffusion $(X_t)$ is a martingale, but as mentioned above, this is
not the case for every speed measure $\nu$.
We characterize here the speed measures for which $X_t$ is a (local)
martingale.

\begin{example}\label{Enotlocalm}
If the speed measure is
given by 
\[\nu(dx)=\left\{\begin{array}{ll}
dx & \mbox{for }x\geq 0\\
0 &  \mbox{for }x< 0,\end{array}\right.\]
then the corresponding process $X$ is a Brownian motion reflected at 0, and,
in particular, not a local martingale. 
\end{example}

This type of reflection at an extreme point of $\supp\nu$ is the
only thing 
that can prevent $X$ from being a local martingale. 
This is shown in the theorem below from
\cite{EH}, here somewhat extended. 

We first give a lemma, essentially saying that stopping $X$ at a point
$a\in\supp\nu$ is the same as stopping $B$ at $a$. (Note that we cannot stop
$X$ at $a\notin\supp\nu$ since $X_t\in\supp\nu$ for all $t\ge0$ by
\refL{Lsupp}.) 

\begin{lemma}
  \label{LM0} 
If $a\in(\xsoo_-,\xsoo_+)\cap\supp\nu$ and $\tau:=\inf\set{t:X_t=a}$, then
$\tau=\gG_{H_a}$. 
If\/ $H_a<\hoo$, then also $\tau<\infty$ and $A_\tau=H_a$, while
$\tau=\infty$ if $H_a>\hoo$.
Furthermore, $\P(H_a<\hoo)>0$ and
$\P(\tau<{t_0})>0$ for any $t_0>0$.
\end{lemma}
\begin{proof}  
  If $H_a<\hoo$, then $\gG_{H_a}<\infty$ by \refL{LH} and thus \refL{LT} 
and the assumption $a\in\supp\nu$ imply
  that $\gG_{H_a+\eps}>\gG_{H_a}$ for every $\eps>0$. Thus, if
  $t=\gG_{H_a}$, then $A_t=H_a$ and $X_t=B_{A_t}=B_{H_a}=a$.
On the other hand, if $t<\gG_{H_a}$, then $A_t<H_a$ by \eqref{rr<} so
$X_t=B_{A_t}\neq a$. Hence, $\tau=\gG_{H_a}$.

If $H_a>\hoo$, then \refL{LH} yields $\gG_{H_a}=\infty$.
\refL{LH} yields further, for any $t<\infty$, $A_t\le\hoo<H_a$ so
$X_t=B_{A_t}\neq a$. Hence, $\tau=\infty=\gG_{H_a}$.

For the final statement, assume that $a\le \xo$, say. Note that the
assumption implies $\xo\notin\suppoo\nu$, so $\xsoo_-<\xo<\xsoo_+$.
Choose $b\in(\xo,\xsoo_+) $. Then $\nu[a,b]<\infty$. Let
$\eps:=t_0/\nu[a,b]$.
(The case $\nu[a,b]=0$ is easy.)
It is easy to see (e.g.\ using excursion theory)
that with positive probability both $\sup_x L^x_{H_a}<\eps$ and
$H_a<H_b$. Thus $B_s\in[a,b)$ for all $s\le H_a$, 
and then $\tau=\gG_{H_a}=\int_a^b L^x_{H_a}\nu(dx)<\eps\nu[a,b]=t_0$.

The claim $\P(H_a<\hoo)>0$ follows from this, using
\refL{LH}, but it is easier to see it directly
 since $\hoo=H_{\set{\xsoo_-,\xsoo_+}}$.
\end{proof}

\begin{theorem}
\label{Tekstrom-hobson}
$(X_t)_{t\ge0-}$ is a local martingale if and only if
either $\supp\nu=\set{x_0}$ (then, trivially, $X_t=\xo$ for all $t$)
or both the following conditions hold:
\begin{romenumerate}
\item\label{eh-}
$\suppoo\nu\cap(-\infty,\xo]\neq\emptyset$ 
or $\supp\nu\cap(-\infty,b]\neq\emptyset$ for all $b<\xo$,
\item\label{eh+} 
$\suppoo\nu\cap[\xo,\infty)\neq\emptyset$ 
or
$\supp\nu\cap[b,\infty)\neq\emptyset$ for all $b>\xo$.
\end{romenumerate}

Moreover, 
if $(X_t)_{0-\le t\le t_0}$ is a local martingale for some $t_0>0$, then
$(X_t)_{t\ge0-}$ is a local martingale.
\end{theorem}

\begin{proof}
If $\supp{\nu}=\set{\xo}$, then $X_t=\xo$ a.s.\ for all $t$ by \refL{Lsupp}.

Suppose now that \ref{eh-} and \ref{eh+} hold.
If $\suppoo\nu\cap(-\infty,\xo]\neq\emptyset$, let 
$a_n^-:=\xsoo_-=\sup\set{x\le \xo:x\in \suppoo\nu}$ for each $n$; otherwise let
$a_n^-$ be a sequence of points in 
$\supp\nu\cap(-\infty,\xo)$ with $a_n^-\downto-\infty$ as \ntoo.
Define a sequence $a_n^+\ge\xo$ similarly, using \ref{eh+}.

Let $H_n:=H_{\set{a_n^-,a_n^+}}$. For each $n$, 
$(B_{t\wedge H_n})_{t\ge0}$ is a bounded martingale. Thus, if $s<t$, then 
$\E(B_{A_t\wedge H_n}\mid\cF_{A_s})=B_{A_s\wedge H_n}$, so 
$(B_{A_t\wedge H_n})_{t\ge0-}$ is a martingale
for the filtration $(\cG_t)=(\cF_{A_t})$.  

Set $T_n:=\gG_{H_n+}=\inf\set{\gG_u:u>H_n}$.
$T_n$ is a stopping time for the filtration $(\cG_{t})$.
Further, since $H_n\le H_{n+1}$ we have $T_n\le T_{n+1}$.

If $T_n<\infty$, then $\gG_{H_n}=T_n<\infty$ and it follows from \refL{LT} that
a.s.\ then $\gG_u>T_n$ for all $u>H_n$, so $A_{T_n}=H_n$;
hence, 
$B_{A_t\wedge H_n}=B_{A_t\wedge A_{T_n}}=B_{A_{t\wedge T_n}}=X_{t\wedge T_n}$.
On the other hand, it $T_n=\infty$, then $A_t\le H_n$ for every $t$, and thus
$B_{A_t\wedge H_n}=B_{A_t}=X_t=X_{t\wedge T_n}$.
In any case, thus
$B_{A_t\wedge H_n}=X_{t\wedge T_n}$ a.s., so 
$(X_{t\wedge T_n})_{t\ge0-}$ is a martingale.

Finally, we verify that $T_n\to\infty$ a.s.\ as \ntoo. 
Recall that $B_{H_n}\in\set{a_n^-,a_n^+}$; we consider for definiteness the
case when $B_{H_n}=a_n^-$, so $H_n=H_{a_n^-}$. By construction, either
$a_n^-\in\suppoo\nu$, and then \refL{LH} implies that
$\gG_{H_{a_n^-}+}=\infty$ a.s., or else $a_n^-\to -\infty$ and then
$H_{a_n^-}\to\infty$ and thus
$\gG_{H_{a_n^-}}\to\infty$ a.s.
In both cases we obtain $T_n=\gG_{H_{a_n^-}+}\to\infty$ a.s.
This completes the proof that $(X_t)_{t\ge0-}$ is a local martingale.

For the converse, we assume that $\supp\nu\neq\set{\xo}$.
Assume that \ref{eh-} fails. (The case when \ref{eh+} fails is symmetric.)
Let $\xx_-:=\inf\set{x:x\in\supp\nu}$. Since \ref{eh-} fails, $\xx_->-\infty$.
We may assume that $\xx_-\le \xo$, since otherwise
\refL{L00} shows that $X_t$ is not a local martingale at $0-$.

Let $T_1=H_{\xx_-}$, the hitting time for $B_u$ of $\xx_-$.
Since \ref{eh-} fails, $[\xx_-,\xo]\cap\suppoo\nu=\emptyset$
so $\xi_-\in(\xsoo_-,\xsoo_+)$ and \refL{LM0} shows that  
with positive probability $T_1=H_{\xx_-}<\hoo$, i.e.,
$B_u$ hits $\xx_-$ before it hits $\suppoo\nu$;
denote this event by $\cE$. 
\refL{LM0} further shows that   
if $\cE$ holds, then there exists $s<\infty$ (viz.\ $s=\gG_{T_1}$)
such that
$X_s=\xx_-$. We want to show that in this case there also exists $t>s$ with
$X_t\neq \xx_-$. 

If $\supp\nu=\set{\xx_-}$, then $\xx_-<\xo$ by our assumption
$\supp\nu\neq\set\xo$, but then $X_t$ is not a local martingale at $0-$ by
\refL{L00}. 
We may thus assume that
there exists another point $x_2\neq \xx_-$ in $\supp\nu$.
Let $T_2$ be the first hitting time of $x_2$ after $T_1$.
We study two cases separately.

If $\gG_{T_2}<\infty$, then 
\refL{LT} implies that a.s.\ 
$\gG_{T_2+\eps}>\gG_{T_2}$ for all $\eps>0$. In this case, if $t=\gG_{T_2}$,
then $A_t=T_2$ and $X_t=B_{T_2}=x_2\neq \xx_-$.

If $\gG_{T_2}=\infty$, then $\aoo=\inf\set{u:\gG_u=\infty}$
is finite.
If $U$ is a neighbourhood of $B_\aoo$, we can find $\eps>0$ such that
$B_u\in U$ for $u\in(\aoo-\eps,\aoo+\eps)$, and thus 
$L_{\aoo+\eps}^x=L_{\aoo-\eps}^x$ for $x\notin U$. Since
$\gG_{\aoo+\eps}-\gG_{\aoo-\eps}=\infty$, it follows that $\nu(U)=\infty$.
Hence $B_\aoo\in\suppoo\nu$, and in particular $B_\aoo\neq \xx_-$. As
$t\to\infty$, $A_t\to\aoo$ and thus $X_t=B_{A_t}\to B_\aoo\neq \xx_-$, so 
in this case,
$X_t\neq \xx_-$ for all large $t$.

Note also that $X_t\ge \xx_-$ for all $t\ge0$ by \refL{Lsupp}. 
If $(X_t)_{t\ge0-}$ is a local martingale, then $Y$ given by $Y_t=X_t-\xx_-$ is thus a 
non-negative local martingale, hence a supermartingale. Therefore zero is absorbing for
$Y$, see e.g.\ \cite[Proposition (3.4)]{RY}.
But this contradicts our
result above which says that with positive probability $Y_t$ first hits zero
and then takes a larger value. 
Hence $X$ is not a local martingale.

For the final statement, it suffices to show that $\P(S_\gd\le t_0)>0$ for 
some $\gd>0$ since the argument above then shows that $(X_t)_{t\le t_0}$ is
not a local martingale.
With $\tau_1=\inf\set{t\ge0:X_t=\xi_-}=\gG_{T_1}$ and 
$\tau_2=\inf\set{t>\tau_1:X_t\neq\xi_-}$ it thus suffices to show
$\P(\tau_2<t_0)>0$. 

By \refL{LM0}, $\P(\tau_1\le t_0/2)>0$. 
Let $\gD:=\tau_2-\tau_1$ and condition on $\cE$. Then $\gD<\infty$ a.s.\ and
the strong Markov property of $B$ implies that $\gD$ is independent of $\tau_1$
and, moreover, that $\P(\gD>a+b\mid\gD>a)=\P(\gD>b)$ for any $a,b>0$ with
$\P(\gD>a)>0$; thus  (conditioned on $\cE$)
either $\gD=0$ a.s.\ or $\gD$ has an exponential distribution;
in both cases  $\P(\gD<t_0/2)>0$.
Hence, $\P(\tau_2<t_0)>0$ which completes the proof.
\end{proof}

Before giving the corresponding characterization for martingales, we give 
some lemmas.

\begin{lemma}
  \label{Lgat}
If $A_t<\hoo$, then $\gG_{A_t}=t$.
Consequently, for all $t\ge0$, 
$\gG_{A_t}=t\wedge\gG_{\hoo}$.
\end{lemma}
\begin{proof}
  If $A_t<\hoo$, then for small $\eps>0$ we have $A_t+\eps<\hoo$ and thus 
$\gG_{A_t+\eps}<\infty$ by \refL{LH}, which by \refR{RRdef}\ref{RR<} yields
$\gG_{A_t}=t$. Furthermore, in this case $t=\gG_{A_t}\le\gG_{\hoo}$ so 
$\gG_{A_t}=t=t\wedge\gG_{\hoo}$.

Since $A_t\le\hoo$ by \refL{LH}, the only remaining case is $A_t=\hoo$.
In this case $\gG_{\hoo}=\gG_{A_t}\le t$ so  
$t\wedge\gG_{\hoo}=\gG_{\hoo}=\gG_{A_t}$.
\end{proof}

\begin{lemma}
  \label{LM1}
For every measurable function $f\ge0$, a.s.\ for every $u\le\hoo$,
\begin{equation}\label{lm1}
  \int_0^u f(B_s)\,d\gG_s
= \intoooo f(x)L_u^x\,\nu(dx).
\end{equation}
\end{lemma}
\begin{proof}
  By monotone convergence it suffices to consider $u<\hoo$, and then, by
  modifying $\nu$ outside the range of $\set{B_s:s\le u}$, it suffices to
  consider locally finite $\nu$. In this case, the result is a consequence
  of the general theory of continuous additive functionals, see
  \cite[Corollary X.(2.13)]{RY}. (It is also easy to make a direct proof in
  this case, by first assuming that $f$ is continuous, and then partitioning
  the interval $[0,u]$ into $N$ subintervals, and for each subinterval
  estimating the change of the difference between the two sides in
  \eqref{lm1}; we omit the details.)
\end{proof}

\begin{lemma}
  \label{LM2}
If $f\ge0$ is a deterministic or random measurable function, then a.s.\ for
every $T\le\hoo$,
\begin{equation}\label{lm2}
\int_0^Tf(s)\,d\gG_s = \int_0^{\gG_T} f(A_t)\,dt.  
\end{equation}
\end{lemma}
\begin{proof}
Consider a fixed $\go$ in our probability space. By a monotone class
argument (or by seeing the two sides of \eqref{lm2} as $\int f\,d\mu_L$
and $\int f\,d\mu_R$  for two finite measures $\mu_L,\mu_R$ on $\ooo$),
it suffices to prove \eqref{lm2} when $f$ is the indicator of an interval
$[0,u)$ for some $u>0$. In this case
$\int_0^Tf(s)\,d\gG_s = \int_0^{T\wedge u}d\gG_s = \gG_{T\wedge u}$
and, by \refRR{RRdef}\ref{RR<},
\begin{equation*}
\int_0^{\gG_T} f(A_t)\,dt
=
\int_0^{\gG_T} \ett{t<\gG_u}\,dt
=\gG_T\wedge \gG_u = \gG_{T\wedge u}.
\qedhere
\end{equation*}
\end{proof}

If $(X_t)_{t\ge0-}$ is a martingale, then $\E X_{t\wedge\tau}=\xo$ for every
$X$-stopping time $\tau$ and every $t\ge0$. This means that $\E B_T=\xo$ for
the $B$-stopping time $T=A_t\wedge A_\tau$. We would like to have $\E
B_T=\xo$ also for other $B$-stopping times $T\le A_t$, not necessarily
obtained by stopping $X$. However, this is not always possible, as is seen
by the following example.
\begin{example}\label{Exo}
  If $\supp\nu=\set{\xo}$, then $X_t=\xo$ for all $t\ge0-$, so $X_t$ is
  trivially a martingale. However, if further $\nu\set{\xo}<\infty$, for
  example if $\nu=\gd_{\xo}$, then for any $t>0$ and $a\neq \xo$,
  $\P(B_{A_t\wedge H_a}=a)=\P(H_a<A_t)>0$; since 
$B_{A_t\wedge H_a}\in\set{a,\xo}$, this implies $\E B_{A_t\wedge H_a}\neq\xo$.
\end{example}

The next lemma shows that \refE{Exo} is the only counterexample.
(The trivial example shows that the lemma is not as trivial as it might look.)

\begin{lemma}
  \label{LM3}
Suppose that $\supp\nu\neq\set{\xo}$ and that $(X_t)_{t\le t_0}$ is a
martingale for some $t_0>0$.
Then, for every stopping time $T\le A_{t_0}$ and every real $a$, 
$\E B_{T\wedge H_a}=\xo$.
\end{lemma}

It will follow from \refT{Tmartingale} and its proof that, more generally,
$\E B_T=\xo$ for any such $T$, so the $H_a$ is not really needed but it
simplifies the proof. 

\begin{proof}
  We suppose that $a<\xo$. (The case $a>\xo$ is similar and $a=\xo$ is
  trivial.)

Suppose first that $a\in\supp\nu$. Then \refL{LT} implies that a.s.\
$\gG_{H_a+\eps}>\gG_{H_a}$ for every $\eps>0$. Hence, if $\tau:=\gG_{H_a}$,
then $A_\tau=H_a$ and $X_\tau=B_{H_a}=a$; moreover, $X_s\neq a$ for
$s<H_a$. In other words, $\tau$ is the $X$-stopping time $\inf\set{s:X_s=a}$.
The assumption that $(X_t)_{t\le t_0}$ is a martingale thus implies, for
$t\le t_0$,
\begin{equation}
  \label{lm3}
\xo=\E X_{t\wedge\tau}
=\E B_{A_t\wedge A_\tau}
=\E B_{A_t\wedge H_a}.
\end{equation}

If $T'$ and $T''$ are two stopping times with $T'\ge T''$, then 
$\E\bigpar{B_{T'\wedge u}|\cF_{T'' }}=B_{T''\wedge u}$
for every $u\ge0$ (see e.g.\ \cite[Theorem 7.29]{Kallenberg}). 
If $T'\le H_a$, then $B_s-a\ge0$ for all $s\le T'$, and
Fatou's lemma yields
$\E(B_{T'}-a\mid\cF_{T''})\le B_{T''}-a$ and thus, by taking the
expectation,
\begin{equation}\label{mag}
  \E B_{T'}\le \E B_{T''}.
\end{equation}
We apply \eqref{mag} first to $A_{t_0}\wedge H_a$ and $T\wedge H_a$ and then
to $T\wedge H_a$ and 0, obtaining
\begin{equation}
 \E B_{A_{t_0}\wedge H_a} \le\E B_{T\wedge H_a} \le \E B_0=x_0,
\end{equation}
and \eqref{lm3} shows that we have equalities. This proves the result when
$a\in\supp\nu$.

In general, by \refT{Tekstrom-hobson}, either there exists some $b<a$ with
$b\in\supp\nu$, or there exists $b\in[a,\xo]$ with $b\in\suppoo\nu$.
In the first case $H_a<H_b$ and in the second case $H_a\ge H_b$ and
\refL{LH} implies that $A_{t_0}\le H_b$ and thus $T\le A_{t_0}\le H_b\le
H_a$. Consequently, in both cases $T\wedge H_a =T\wedge H_a \wedge H_b$, and
the result follows from the case just proven applied to $T\wedge H_a $ and $b$.
\end{proof}

\begin{theorem}\label{Tmartingale} 
The following are equivalent, for any $t_0>0$.
\begin{romenumerate}
\item \label{tmm}
$(X_t)_{t\ge0-}$ is a martingale.
\item \label{tmt0}
$(X_t)_{0-\le t\le t_0}$ is a martingale.
\item \label{tmh}
$\supp\nu=\set\xo$ or
$\xo\in\suppoo\nu$ or
  \begin{equation}\label{tm}
\int_{\xo}^\infty(1+ |x|)\nu(dx)	
=
\int_{-\infty}^{\xo}(1+ |x|)\nu(dx)	
=\infty.
  \end{equation}

\item \label{tmh0}
$\supp\nu=\set\xo$ or
$\xo\in\suppoo\nu$ or
  \begin{equation}\label{tm0}
\int_{\xo}^\infty |x-\xo|\,\nu(dx)	
=
\int_{-\infty}^{\xo} |x-\xo|\,\nu(dx)	
=\infty.
  \end{equation}
\end{romenumerate}
\end{theorem}

\begin{remark}
For a related result for non-negative diffusion processes, see \cite{DS}.
\end{remark}

\begin{proof}
The result is trivial if $\supp\nu=\set\xo$ or $\xo\in\suppoo\nu$. 
We may thus assume $\supp\nu\neq\set\xo$ and
$-\infty\le\xsoo_-<\xo<\xsoo_+\le\infty$. The equivalence of \eqref{tm} and
\eqref{tm0} then is elementary.

Let
\begin{align}
  \gf(x)&:=
  \begin{cases}
	\phantom{{}-{}}2\int_{[\xo,x)}|y|\,\nu(dy), & x\ge \xo,\\
  {}-2\int_{[x,\xo)}|y|\,\nu(dy), & x< \xo,
  \end{cases}
\intertext{and}
  \psi(x)&:=
  \begin{cases}
	\phantom{{}-{}}\int_{\xo}^x\gf(y)\,dy, & x\ge \xo,\\
	{}-\int_x^{\xo}\gf(y)\,dy, & x< \xo.
  \end{cases}
\end{align}
Then $\psi$ is a non-negative convex function on $(\xsoo_-,\xsoo_+)$
with left derivative $\gf$ and thus second derivative (in distribution
sense)
$\psi''(x)=2|x|\nu(dx)$.

By the It\^o--Tanaka formula 
\cite[Theorem VI.(1.5)]{RY}
and \refL{LM1}, for $u\le\hoo$,
\begin{equation}\label{m1}
  \begin{split}
\psi(B_u)
&=\int_0^u\gf(B_s)\,dB_s +
\intR L_u^x\,|x|\,\nu(dx)	
=\int_0^u\gf(B_s)\,dB_s + \int_0^u|B_s|\,d\gG_s.
  \end{split}
\end{equation}

Let $H:=H_{\set{a,b}}$ where $\xsoo_-<a<\xo<b<\xsoo_+$. Then $H<\hoo$. Further,
$\gf(B_s)$ is bounded for $s\le H$, and  consequently   
$\int_0^{u\wedge H}\gf(B_s)\,dB_s $, $u\ge0$, is a martingale.
Hence, for every bounded stopping time $T\le H$,
$\E\int_0^{T}\gf(B_s)\,dB_s=0 $, and \eqref{m1} yields
\begin{equation}\label{m2}
\E(\psi(B_{T}))
= \E \int_0^{T}|B_s|\,d\gG_s.
\end{equation}
Moreover, for any stopping time $T\le H$, we can apply \eqref{m2} to 
$T\wedge u$ and let $u\to\infty$. Since $T\le H$ we have 
$B_{T\wedge u}\in[a,b]$, and thus  $\psi(B_{T\wedge u})$ is uniformly
bounded; hence
dominated convergence on the left-hand side and monotone convergence on the
right-hand side shows that \eqref{m2} holds for any stopping time $T\le H$.

We apply \eqref{m2} first to $T=A_r\wedge H\wedge u$ for some $r,u\ge0$.
Thus, using \refL{LM2} and $\gG_{A_r}\le r$, see \refRR{RRdef}\ref{RR<},
\begin{equation}\label{m3}
  \begin{split}
\E&(\psi(B_{A_r\wedge H\wedge u}))
= \E \int_0^{A_r\wedge H\wedge u}|B_s|\,d\gG_s
= \E \int_0^{A_r\wedge H\wedge u}|B_{s\wedge H\wedge u}|\,d\gG_s
\\
&= \E \int_0^{\gG_{A_r\wedge H\wedge u}}|B_{A_t\wedge H\wedge u}|\,dt
\le \E \int_0^{r}|B_{A_t\wedge H\wedge u}|\,dt
=
 \int_0^{r}\E|B_{A_t\wedge H\wedge u}|\,dt.	
  \end{split}
\end{equation}

Suppose first that \eqref{tm0} holds. 
By translation we may assume that $\xo=0$.
Then $\gf(x)\to\pm\infty$ as
$x\to\pm\infty$, and thus $\psi(x)/|x|\to\infty$ as $x\to\pm\infty$.
In particular, $\psi(x)\ge|x|$ for large $|x|$, and thus $|x|\le\psi(x)+C$
for some constant $C$.
Consequently, \eqref{m3} yields
\begin{equation*}
\E|B_{A_r\wedge H\wedge u}|
\le C+\E(\psi(B_{A_r\wedge H\wedge u}))
\le
C+ \int_0^{r}\E|B_{A_t\wedge H\wedge u}|\,dt.	
\end{equation*}
Note further that $B_{A_t\wedge H\wedge u}$ is bounded by $\max\set{|a|,b}$,
so the expectations here are finite and bounded. We can thus apply
Gronwall's Lemma \cite[Appendix \S1]{RY} and conclude
\begin{equation}
  \E|B_{A_t\wedge H\wedge u}| \le Ce^t.
\end{equation}
Using this in \eqref{m3} again yields the sharper estimate
\begin{equation}
  \E(\psi(B_{A_t\wedge H\wedge u})) \le Ce^t.
\end{equation}
Now let $a\to\xsoo_-$ and $b\to\xsoo_+$; then
$H\to\hoo=H_{\set{\xsoo_-,\xsoo_+}}$. 
Since $A_t\le \hoo$ by \refL{LH}, 
$B_{A_t\wedge H\wedge u}\to B_{A_t\wedge u}$, and Fatou's lemma yields
\begin{equation}\label{tm2}
  \E(\psi(B_{A_t\wedge u})) \le Ce^t.
\end{equation}

Since $\psi(x)/|x|\to\infty$ as $x\to\pm\infty$, \eqref{tm2} implies that
for a fixed $t$, the random variables $B_{A_t\wedge u}$, $u\ge 0$, 
are uniformly integrable \cite[Theorem 5.4.3]{Gut}.
Moreover, $B_{A_t\wedge u}$, $u\ge 0$, is a martingale, and since it is
uniformly integrable, 
$B_{A_t\wedge u}=\E(B_{A_t}\mid\cF_u)$, for all $u\ge0$, and further, see
\cite[Theorem 7.29]{Kallenberg}, for any $ s\le t$,
\begin{equation}
\E(B_{A_t}\mid\cF_{A_s})
=
B_{A_t\wedge A_s}=
B_{A_s}.
\end{equation}
Hence $X_t=B_{A_t}$ is a martingale when \eqref{tm0} holds, so
\ref{tmh0}$\implies$\ref{tmm}$\implies$\ref{tmt0}. 

Conversely, suppose that \ref{tmt0} holds but
\eqref{tm} fails; by symmetry we may assume that
\begin{equation}
  \label{contra}
\int_{\xo}^\infty (1+|x|)\,\nu(dx)	<\infty.
\end{equation}
In particular, this shows that $\suppoo\nu\cap(\xo,\infty)=\emptyset$, so
$\xsoo_+=\infty$. 
By translation, we may now also assume that
$\xsoo_-<0<\xo$, and 
we now take $a=0$ and any $b>\xo$.
Thus $H=H_{\set{a,b}}=H_0\wedge H_b$.
%
We  apply \eqref{m2}  with $T=A_t\wedge H$.
Noting that $H\le H_0$ and $B_s\ge0$ for $s\le H$, we obtain, using also
\refL{LM2},  
\begin{equation}
  \label{m5a}
  \begin{split}
\E \psi\bigpar{B_{A_t\wedge H_0\wedge H_b}}
&=
\E \int_0^{A_t\wedge H}	 B_s\,d\gG_s
=
\E \int_0^{A_t\wedge H}	 B_{s\wedge H_0}\,d\gG_s
\\&
=
\E \int_0^{\gG_{A_t\wedge H}} B_{A_r\wedge H_0}\,dr
.
  \end{split}
\end{equation}
Since $H<\hoo$,  \refL{Lgat} shows that
\begin{equation*}
  \gG_{A_t\wedge H}=  \gG_{A_t}\wedge\gG_{H}
= t\wedge\gG_{\hoo}\wedge\gG_{H}
= t\wedge\gG_{H}
= t\wedge\gG_{H_0}\wedge\gG_{H_b}.
\end{equation*}
 Further, if $r\ge \gG_{H_0}$ then $A_r\ge H_0$ by \eqref{rr<}
and thus 
$B_{A_r\wedge  H_0}=B_{H_0}=0$. 
Hence, \eqref{m5a} yields
\begin{equation}
  \label{m5b}
  \begin{split}
\E \psi\bigpar{B_{A_t\wedge H_0\wedge H_b}}
=
\E \int_0^{t\wedge\gG_{H_0}\wedge \gG_{H_b}}	 B_{A_r\wedge H_0}\,dr
=
\E \int_0^{t\wedge \gG_{H_b}}	 B_{A_r\wedge H_0}\,dr
.
  \end{split}
\end{equation}

Let $C':=2\int_0^\infty y\,\nu(dy)$. Then $|\gf(x)|\le C'$ for $x\ge 0$, and
thus, for any $x\ge 0$,
\begin{equation}
  \psi(x)\le \psi(0)+C'x\le \psi(\xo)+C'\xo+C'x=C'\xo+C'x.
\end{equation}
Thus, for $t=t_0$,
the left-hand side of \eqref{m5b} is at most, using \refL{LM3}, 
\begin{equation}
\E \bigpar{C'\xo+C'B_{A_{t}\wedge H_0\wedge H_b}}  
=2C'\xo.
\end{equation}
Consequently \eqref{m5b} yields
\begin{equation}\label{ol}
  2C'\xo \ge \E \int_0^{t_0\wedge \gG_{H_b}}	 B_{A_r\wedge H_0}\,dr.
\end{equation}
Letting $b\to\infty$ we have $H_b\to\infty$ and $\gG_{H_b}\to\infty$, and
thus \eqref{ol} yields by monotone convergence, using \refL{LM3} again,
\begin{equation*}
  2C'\xo \ge \E \int_0^{t_0} B_{A_r\wedge H_0}\,dr
=
 \int_0^{t_0} \E B_{A_r\wedge H_0}\,dr
=t_0\xo.
\end{equation*}
Consequently,
\begin{equation}\label{jb}
  t_0 \le 2C'= 4\int_0^\infty y\,\nu(dy).
\end{equation}
This shows that \ref{tmm}$\implies$\ref{tmh}, since we then may choose $t_0$
arbitrarily large in \eqref{jb} and obtain a contradiction.
To extend this to \ref{tmt0} with a given $t_0$, we argue as follows,
assuming \ref{tmt0} and \eqref{contra}.

We have derived \eqref{jb} under the assumption $\xsoo_-<0<\xo$. By
translation, we have in general, for any $\xo$ and $\xsoo_-$, and any
$a\in(\xsoo_-,\xo)$, 
\begin{equation}
  t_0 \le  4\int_a^\infty y\,\nu(dy).
\end{equation}
Letting $a\to\xo$ yields
\begin{equation}\label{jc}
  t_0 \le 4\int_{\xo}^\infty y\,\nu(dy).
\end{equation}

Take any $z>\xo$ with $z\in\supp\nu$, and let $\tau:=\gG_{H_z}$;
by \refL{LM0}, $\tau=\inf\set{t:X_t=z}$, so $\tau$ is an $X$-stopping time.
Condition on the event
$\cE
:=\set{\tau\le t_0/2}$; we have $\P(\cE)>0$
by \refL{LM0}.
On the event $\cE$, $B_u'=B_{H_z+u}$ is a Brownian motion starting at $z$,
and the processes corresponding to $\gG$, $A$ and $X$ defined by $B'$ are
$\gG'_u=\gG_{H_z+u}-\gG_{H_z}=\gG_{H_z+u}-\tau$,
$A'_t=A_{t+\tau}-H_z$ and
$X'_t=B'_{A'_t}=X_{t+\tau}$.
Since $\tau$ is an $X$-stopping time, on $\cE$, 
$X'_t$ is a martingale for $0\le t\le t_0/2$, and also at $t=0-$ since
$X'_{0}=z=X'_{0-}$, cf.\ \refL{L0}.
Consequently, we may apply the result above to $X'$, and \eqref{jc} yields 
\begin{equation}\label{jd}
 \frac{t_0}2 \le 4\int_{z}^\infty y\,\nu(dy).
\end{equation}
However, if \eqref{contra} holds, then the right-hand side of \eqref{jd}
tends to 0 as $z\to\infty$, which is a contradiction.
This contradiction shows that \eqref{contra} cannot hold, and thus
\ref{tmt0}$\implies$\ref{tmh}.
\end{proof}

Note that the case $\xo\in\suppoo\nu$ is not redundant in \ref{tmh}. 
An example is given by $\xo=0$ and $\nu(dx)=dx/x$ for $x>0$; then \eqref{tm}
fails because the second integral vanishes, nevertheless $X_t=\xo=0$ for all
$t$, which trivially is a martingale.

\section{Open problems}
\label{openproblems}

\subsection{Weakening the assumptions}
\refT{cont} shows the existence of a generalised diffusion solving the inverse problem
provided the given distribution has finite expectation. 
Does the result hold without this assumption (allowing local martingales instead
of only martingales)?

\subsection{Uniqueness}
\refT{cont} shows 
the existence of a generalised diffusion, defined by a
speed measure $\nu$, with a given distribution of $X_1$ (assuming $\E|X_1|$
is finite).
Is the  measure $\nu$ unique?
Strictly speaking it is not: 
\refL{LB} shows that we will never leave the interval
$[\xsoo_-,\xsoo_+]$, so any change of $\nu$ outside this interval (assuming at
least one of $\xsoo_-$ and $\xsoo_+$ is finite) will not
affect $X_t$ at all. We may normalize $\nu$ by first taking the restriction
$\nu_0$ to $(\xsoo_-,\xsoo_+)$; if $\xsoo_-$ is finite but
$\xsoo_-\notin\suppoo\nu_0$ we also add an infinite point mass at $\xsoo_-$,
and similarly for $\xsoo_+$.
The question is then: Given the distribution of $X_1$, is there a unique
corresponding normalized speed measure such that the corresponding process
$(X_t)_{t\le1}$ is a martingale? 


The problem of uniqueness is also open in the discrete case, i.e. 
is the mapping $G:\bn\to\pin$ 
studied in \refS{discrete} an injection?
(As remarked in the proof, this holds for $n=1$.)

\subsection{Relations between $\mu$ and $\nu$}
Find relations between the speed measure $\nu$ (assumed normalized as above)
and the distribution
$\mu$ of $X_1$. For example, if $\mu$ has a point mass at $x$, does $\nu$
also have a point mass there? Does the converse hold? If $\nu$ is absolutely
continuous, is $\mu$ too? Does the converse hold?
For which $\nu$ does $\mu$ have a finite second moment?

\begin{ack}
  We thank Tobias Ekholm for helpful comments. 
\end{ack}


\begin{thebibliography}{99999}

\bibitem{Bredon}
Bredon, G. E.,
\emph{Topology and Geometry},
Springer-Verlag, New York, 1993.

\bibitem{CH}
Cox, A. and Hobson, D.,
Local martingales, bubbles and option prices. 
{\em Finance Stoch}. 9 (2005), 477--492. 

\bibitem{CHO}
Cox, A., Hobson, D. and Ob\l\'oj, J., 
Time homogeneous diffusions with a given marginal at a 
random time. arXiv:0912.1719. 
To appear in \emph{ESAIM: Probability and Statistics} (special issue in 
honour of Marc Yor).  

\bibitem{DS}
Delbaen, F. and Shirakawa, H., No arbitrage condition for positive diffusion 
price processes. {\em Asia-Pacific Financial Markets} 9 (2002), 159--168.

\bibitem{D}
Dupire, B.,
Pricing with a smile. {\em Risk} 7 (1994), 18--20.

\bibitem{EH}
Ekstr\"om, E. and Hobson, D., 
Recovering a time-homogeneous 
stock price process from perpetual option prices. To appear in 
{\em Ann. Appl. Probab.} (2011).

\bibitem{ET}
Ekstr\"om, E. and Tysk, J. Bubbles, convexity and the Black-Scholes equation. 
{\em Ann. Appl. Probab.} 19 (2009), no. 4, 1369--1384.

\bibitem{Gut}
Gut, A.,
\emph{Probability: A Graduate Course},
Springer, New York, 2005.
Corrected 2nd printing 2007.

\bibitem{JT}
Jiang, L. and Tao Y., 
Identifying the volatility of underlying assets from option prices.
{\em Inverse Problems} 17 (2001), 137--155.

\bibitem{Kallenberg}
Kallenberg, O.,
\emph{Foundations of Modern Probability.}
2nd ed., Springer, New York, 2002. 

\bibitem{K}
Knight, F. B., 
Characterization of the Levy measures of inverse local times of gap
diffusion. 
In: {\em Seminar on Stochastic Processes}, 1981 Birkh\"auser, Boston, Mass.,
53--78, 1981.

\bibitem{KW}
Kotani, S. and Watanabe, S., 
Krein's spectral theory of strings and generalized diffusion processes. 
In: {\em Functional analysis in Markov processes}, 
Lecture Notes in Math. 923, Springer, Berlin, 1982, 235--259.


\bibitem{MPeres}
M\"orters, P. and Peres, Y.,
\emph{Brownian Motion},
Cambridge Univ. Press, Cambridge, 2010.

\bibitem{RY}  
Revuz, D. and Yor, M.,
\emph{Continuous Martingales and Brownian Motion},
3rd ed.,
Springer, Berlin, 1999.




\end{thebibliography}
\end{document}